\documentclass[12pt,a4paper]{amsart}
\usepackage[utf8]{inputenc}
\usepackage[top=3cm, bottom=3.5cm, left=2.5cm, right=2.5cm]{geometry}
\usepackage{amsmath,amsthm,amsfonts,amstext,amssymb}
\usepackage{mathrsfs}
\usepackage{enumerate}
\usepackage{hyperref}
\usepackage{dsfont}
\usepackage{color}
\usepackage{xcolor}
\hypersetup{
    colorlinks,
    linkcolor={blue!50!black},
    citecolor={blue!50!black},
    urlcolor={blue!80!black}
}

\def\R{{\mathbb{R}}}

\def\Z{{\mathbb{Z}}}

\newtheorem{theorem}{Theorem}[section]
\newtheorem{corollary}[theorem]{Corollary}
\newtheorem{lemma}[theorem]{Lemma}
\newtheorem{proposition}[theorem]{Proposition}

\theoremstyle{definition}

\newtheorem{definition}[theorem]{Definition}
\newtheorem{remark}[theorem]{Remark}

\numberwithin{equation}{section}

\begin{document}

\title[Indistinguishability of unbounded components in Boolean models]{
Indistinguishability of unbounded components in the occupied and vacant sets of Boolean models on symmetric spaces}
\author{Yingxin Mu}
\address{
  Yingxin Mu,
  University of Leipzig, Institute of Mathematics,
  Augustusplatz 10, 04109 Leipzig, Germany.
}
\email{yingxin.mu@uni-leipzig.de}

\author{Artem Sapozhnikov}
\address{
  Artem Sapozhnikov,
  University of Leipzig, Institute of Mathematics,
  Augustusplatz 10, 04109 Leipzig, Germany.
}
\email{artem.sapozhnikov@math.uni-leipzig.de}

\begin{abstract}
We study Boolean models on Riemannian symmetric spaces driven by homogeneous insertion- or deletion-tolerant point processes. We prove that in both the set covered by the balls (the occupied set) and its complement (the vacant set), one cannot distinguish unbounded components from each other by any isometry invariant component property. This implies the uniqueness monotonicity for the occupied and vacant sets of Poisson-Boolean models and an equivalence of non-uniqueness to the decay of connectivity for both sets. These results are continuum analogues of those by Lyons and Schramm \cite{LS-Indistinguishability}. 
However, unlike the proof of the indistinguishability in \cite{LS-Indistinguishability}, our proof does not rely on transience of unbounded components. 
We also prove the existence of a percolation phase transition for independent Poisson-Boolean model on unbounded connected components of both occupied and vacant sets and show transience of a random walk on the occupied set. Apart from some technical differences, we treat the occupied and the vacant sets of Boolean models within a single framework. 
\end{abstract}

 
\maketitle

\section{Introduction}

The \emph{Boolean model} in a metric space $\mathsf X$ is an ensemble of closed balls with independent identically distributed random radii centered at the points of a point process. It is a fundamental model in percolation theory, stochastic geometry, material sciences and telecommunications, see e.g.\ \cite{DousseThiranHasler,Hermann,HugLastWeil, JahnelKoenig,MeesterRoy}. We refer to \cite{MeesterRoy} for a mathematical introduction to the Boolean model. 
Each Boolean model induces a partition of $\mathsf X$ into the \emph{occuped set}, the subset of $\mathsf X$ covered by the balls, and the \emph{vacant set}. 
Connectivity properties of the occupied and vacant sets of Boolean models driven by Poisson point processes in Euclidean space $\R^d$ have been extensively studied. In particular, it is known when both sets undergo a non-trivial percolation phase transition in the intensity of the Poisson process (see e.g.\ \cite{Gouere08, ATT18, Pen18}) and that each set contains at most one unbounded connected component \cite{MR-Uniqueness}. 

\smallskip

The study of connectivity properties in the Poisson-Boolean model on more exotic spaces than $\R^d$ was inititated by Tykesson in \cite{Tykesson-H2, Tykesson-MofU}. 
He proved the existence of a percolation phase transition in the Poisson-Boolean model with constant radii on the hyperbolic space $\mathbb H^d$ and showed the existence of 
a regime with infinitely many unbounded components.
Thus, the following question is natural to ask:
can the unbounded connected components in the occupied (resp.\ the vacant) set be very different from each other? For example, can unbounded components with different volume growth or different densities coexist? 
In a discrete setting of percolation on graphs, Lyons and Schramm \cite{LS-Indistinguishability} proved a remarkable result that one cannot distinguish infinite clusters from each other by any invariant property.
One of the main results of our paper is an analogue of the Lyons-Schramm theorem 
for the occupied and vacant sets in a class of Boolean models.

\smallskip

Let us first introduce the model. 
A simply connected and complete Riemannian manifold $\mathsf X$ is called a \emph{(Riemannian)  symmetric space} if for every $x\in\mathsf X$, the map that reverses geodesics passing through $x$ is an isometry of $\mathsf X$; see e.g.\ \cite{jost}. 
Euclidean and hyperbolic spaces are examples of symmetric spaces, products of symmetric spaces are symmetric. 
Any symmetric space is homogeneous, that is, its isometry group acts transitively on it. 
Throughout this paper, $\mathsf X$ is a non-compact symmetric space with volume measure $\mu_\mathsf X$ and a marked origin $0\in\mathsf X$. 

\smallskip

We denote by $\mathsf M(\mathsf X)$ the space of simple counting measures on $\mathsf X$. 
A random point measure $\omega$ on $\mathsf M(\mathsf X)$ is \emph{insertion-tolerant} if for all bounded $B\in\mathscr B(\mathsf X)$, the law of $\omega+\delta_X$ is absolutely continuous with respect to the law of $\omega$, where $X$ is independent from $\omega$ and uniformly distributed in $B$. A point measure $\omega$ is \emph{deletion-tolerant} if for all bounded $B\in\mathscr B(\mathsf X)$, the law of $\omega|_{B^c}$ is absolutely continuous with respect to the law of $\omega$. 
If a random point measure is insertion- resp.\ deletion-tolerant, then we also call its law insertion- resp.\ deletion-tolerant. 
These definitions of insertion- and deletion-tolerant point measures were proposed by Holryod and Soo in \cite{HolroydSoo}.\footnote{Although the paper is written in the setting of $\R^d$, the definitions and results extend to symmetric spaces, see \cite[Remark~1]{HolroydSoo}.} 
The Poisson point measure on $\mathsf X$ with intensity $\lambda\mu_\mathsf X$ is a natural example of insertion- and deletion-tolerant point measure. Another interesting example is the Gaussian zero process in the hyperbolic plane, which is also both insertion- and deletion-tolerant, see \cite{PeresVirag,HKPV-ZerosOfGAF} and \cite[Proposition 13]{HolroydSoo}.

\smallskip

In this paper, we study Boolean models on symmetric spaces driven by homogeneous insertion- or deletion-tolerant point processes. 
While all our results hold for Boolean models with general i.i.d.\ random radii, for clarity, we restrict the presentation to the case when all the balls have the same unit radius. 
In Section~\ref{sec:random-radii}, we state the results in the general case (Theorem~\ref{thm:main-iid}) and show where and how to adjust the proofs for the constant radii so that they work for general i.i.d.\ random radii.
Each $\omega\in \mathsf M(\mathsf X)$ induces the closed subset $\mathcal O$ of $\mathsf X$, defined by
\[
\mathcal O = \mathcal O(\omega) = \bigcup\limits_{x\in\mathrm{supp}(\omega)}\mathsf B(x,1),
\]
called the \emph{occupied set}. Its complement $\mathcal V = \mathsf X\setminus \mathcal O$ is called the \emph{vacant set}. 
In this paper, we are interested in properties of unbounded connected components of the occupied and the vacant sets. The number of unbounded components can be either $0$, $1$ or infinity (see Lemma~\ref{l:number-of-infinite-clusters-occupied}). Our main results concern the regime when the number of unbounded components is infinite. 
Let us also point out that apart from some technical differences, we treat the occupied and the vacant sets within a single framework. 

\smallskip

We now describe our results. 
A set $\mathcal A\in \mathscr B(\mathsf X)\otimes \mathscr M(\mathsf X)$ is called \emph{occupied component property} if $(x,\omega)\in\mathcal A$ implies that $(x',\omega)\in\mathcal A$ for all $x'$ connected to $x$ in $\mathcal O(\omega)$ 
and it is called \emph{vacant component property} if $(x,\omega)\in\mathcal A$ implies that $(x',\omega)\in\mathcal A$ for all $x'$ connected to $x$ in $\mathcal V(\omega)$. 
A component property $\mathcal A$ is isometry invariant, if $(x,\omega)\in\mathcal A$ implies that $(\gamma x,\gamma\omega)\in\mathcal A$ for all isometries $\gamma$ of $\mathsf X$. Here are some examples of isometry invariant vacant component properties: $(x,\omega)\in\mathcal A_1$ if the connected component of $x$ in $\mathcal V(\omega)$ has infinite volume, $(x,\omega)\in\mathcal A_2$ if the density of the connected component of $x$ in $\mathcal V(\omega)$ is zero. 

\begin{theorem}[Indistinguishability of unbounded components]\label{thm:Indistinguishability}
Let $\mathbf P$ be an isometry invariant probability measure on $\mathsf M(\mathsf X)$. 
\begin{itemize}\itemsep4pt
\item[(a)]
Let $\mathcal A$ be an isometry invariant occupied component property. 
If $\mathbf P$ is insertion-tolerant, then $\mathbf P$-almost surely, either all unbounded occupied components have property $\mathcal A$ or none of them do.
\item[(b)]
Let $\mathcal A$ be an isometry invariant vacant component property.
If $\mathbf P$ is deletion-tolerant and 
the expected number of connected components in $\mathcal V\cap\mathsf B(0,1)$ is finite, then $\mathbf P$-almost surely, either all unbounded vacant components have property $\mathcal A$ or none of them do.
\end{itemize}
\end{theorem}
Theorem~\ref{thm:Indistinguishability} is a continuum analogue of the indistinguishability theorem of Lyons and Schramm \cite{LS-Indistinguishability}. Several other extensions and generalizations of their result were obtained in the settings of Bernoulli percolation \cite{AldousLyons,Martineau,Tang} and uniform spanning forests \cite{HutchcroftNachmias,Hutchcroft,Timar}. 

\smallskip

If $\mathsf X$ is a hyperbolic space, then it is possible to replace the assumption on the number of vacant components in a ball in Theorem~\ref{thm:Indistinguishability} by a more natural moment assumption on the number of points of the point process in a ball. 

\begin{proposition}\label{prop:Indistinguishability-hyperbolic}
If $\mathsf X$ is the $d$-hyperbolic space $\mathbb H^d$ and 
$\mathbf E\big[\omega(\mathsf B(0,2))^d\big]<\infty$, then the expected number of 
connected components in $\mathcal V\cap\mathsf B(0,1)$ is finite. 
\end{proposition}

The moment assumption of Proposition~\ref{prop:Indistinguishability-hyperbolic} is obviously satisfied by any homogeneous Poisson point process on $\mathbb H^d$, but it is also satisfied by the Gaussian zero process in the hyperbolic plane, see \cite[Theorem~3.2.1]{HKPV-ZerosOfGAF}.

\smallskip

In their paper, Lyons and Schramm discussed several implications of the cluster indistinguishability. Similar implications can be derived in our setting. For a subset $\mathcal S$ of $\mathsf X$, we denote by $N_\mathcal S$ the number of unbounded connected components in $\mathcal S$. 

\begin{theorem}[Uniqueness monotonicity]\label{thm:monotonicity-of-uniqueness}
Let $\mathbf P_\lambda$ be the law of a Poisson point measure on $\mathsf X$ with intensity $\lambda\mu_\mathsf X$. Let $\lambda_1<\lambda_2$. 
\begin{itemize}\itemsep4pt
 \item[(a)]
 If $\mathbf P_{\lambda_1}[N_\mathcal O = 1]=1$ then $\mathbf P_{\lambda_2}[N_\mathcal O=1]=1$. 
 \item[(b)]
 If the expected number of connected components in $\mathcal V\cap\mathsf B(0,1)$ is finite $\mathbf P_{\lambda_1}$-almost surely, then $\mathbf P_{\lambda_2}[N_\mathcal V=1]=1$ implies $\mathbf P_{\lambda_1}[N_\mathcal V=1]=1$. 
\end{itemize}
\end{theorem}
Part (a) of Theorem~\ref{thm:monotonicity-of-uniqueness} is not new. It was proven by Tykesson in \cite{Tykesson-MofU} using two different proofs, adapting from discrete setting the arguments of H\"aggstr\"om and Peres \cite{MofU-HaggstromPeres} and Schonmann \cite{MofU-Schonmann}. (In fact, his second proof applies when $\mathsf X$ is a homogeneous space, just as the proof of Schonmann applies to arbitrary transitive graphs.) 
However, both proofs exploit crucially certain sequential revealments of connected components and do not apply to the vacant set. Lyons and Schramm \cite{LS-Indistinguishability} observed that the uniqueness monotonicity follows almost directly from the indistinguishability of infinite components, and we prove Theorem~\ref{thm:monotonicity-of-uniqueness} by adapting their argument to our setting. In particular, parts (a) and (b) are proved in essentially the same way.

\smallskip

The next implication of Theorem~\ref{thm:Indistinguishability} is a relation between non-uniqueness of unbounded components and the connectivity decay; it is a continuum analogue of \cite[Theorem~4.1]{LS-Indistinguishability}.
For a random subset $\mathcal S$ of $\mathsf X$, we denote by $\tau_\mathcal S(x,x')$ the probability that $x$ and $x'$ are in a same connected component of $\mathcal S$.  

\begin{theorem}\label{thm:connectivity}
Let $\mathbf P$ be an isometry invariant ergodic probability measure on $\mathsf M(\mathsf X)$. 
\begin{itemize}\itemsep4pt
\item[(a)]
If $\mathbf P$ is insertion-tolerant and $N_\mathcal O=\infty$ almost surely, then 
$\inf_{x,x'\in\mathsf X}\tau_\mathcal O(x,x') =0$.
\item[(b)]
If $\mathbf P$ is deletion-tolerant, 
the expected number of connected components in $\mathcal V\cap\mathsf B(0,1)$ is finite and $N_\mathcal V=\infty$ almost surely, then $\inf_{x,x'\in\mathsf X}\tau_\mathcal V(x,x') =0$.
\end{itemize}
\end{theorem}
The result of Theorem~\ref{thm:connectivity} is sharp for positively correlated systems. Indeed, if a random subset $\mathcal S$ of $\mathsf X$ with an isometry invariant law satisfies the FKG-inequality, more precisely, if $\mathsf P[\text{$C_\mathcal S (x)$ and $C_\mathcal S(x')$ are unbounded}]\geq 
\mathsf P[\text{$C_\mathcal S (x)$ is unbounded}]\,\mathsf P[\text{$C_\mathcal S (x')$ is unbounded}]$, where $C_\mathcal S(x)$ stands for the connected component of $x$ in $\mathcal S$, then
\[
\tau_\mathcal S(x,x')\geq\mathsf P[\text{$C_\mathcal S (x)$ and $C_\mathcal S(x')$ are unbounded}]\geq\mathsf P[\text{$C_\mathcal S (0)$ is unbounded}]^2>0,
\]
provided that $N_\mathcal S=1$ almost surely. In particular, this is the case for the Poisson-Boolean model (see \cite[Section~2.3]{MeesterRoy}). 

\smallskip

Tools that we develop to prove Theorem~\ref{thm:Indistinguishability} allow us to establish further properties of unbounded connected components in the regime of non-uniqueness. In the next theorem, which follows from Theorem~\ref{thm:percolation}, we consider an independent Poisson-Boolean model restricted to the occupied or to the vacant set and prove that it exhibits a percolation phase transition in each of infinitely many unbounded components. 
\begin{theorem}\label{thm:percolation-intro}
Let $\omega$ be a random point measure on $\mathsf X$ with an ergodic isometry invariant law and $\eta_\lambda$ an independent Poisson point measure on $\mathsf X$ with intensity $\lambda\mu_\mathsf X$. 

\smallskip

There exists $\lambda_*<\infty$ such that for all $\lambda>\lambda_*$, 
\begin{itemize}
\item[(a)]
if $\omega$ is insertion-tolerant and $N_{\mathcal O(\omega)}=\infty$ almost surely, then almost surely for every unbounded connected component $\mathcal C$ of $\mathcal O(\omega)$, 
the occupied set of the Boolean model on $\mathcal C$ driven by $\eta_\lambda|_\mathcal C$ contains an unbounded connected component;
\item[(b)]
if $\omega$ is deletion-tolerant, the expected number of connected components in $\mathcal V(\omega)\cap\mathsf B(0,1)$ is finite and $N_{\mathcal V(\omega)}=\infty$ almost surely, then almost surely for every unbounded connected component $\mathcal C$ of $\mathcal V(\omega)$, the occupied set of the Boolean model on $\mathcal C$ driven by $\eta_\lambda|_\mathcal C$ contains an unbounded connected component.
\end{itemize}
\end{theorem}

The next theorem, which follows from Theorem~\ref{thm:rw-transience}, states that any 
of infinitely many unbounded occupied components of an insertion-tolerant Boolean model is transient. 
\begin{theorem}[Transience]\label{thm:rw-transience-intro}
Let $\omega$ be a random point measure on $\mathsf X$ with an isometry invariant insertion-tolerant law. Let $G$ be the graph of the Boolean model driven by $\omega$, that is, the vertex set of $G$ is the support of $\omega$ and its edge set is the set of all pairs $x,x'\in\mathrm{supp}(\omega)$ with $\mathsf B(x,1)\cap\mathsf B(x',1)\neq\emptyset$. If $G$ has infinitely many infinite connected components almost surely, then they are all transient almost surely. 
\end{theorem}

\smallskip

We now comment on the proof of Theorem~\ref{thm:Indistinguishability}. 
Although we do implement a version of the main idea of Lyons and Schramm \cite{LS-Indistinguishability}, 
the crucial difference is that we do not use transience of unbounded components. 
In particular, when applied in the setting of percolation on graphs, our approach gives a shortcut for the proof of Lyons and Schramm, see discussion below and Remark~\ref{rem:indistinguishability-graphs}.

There are only a few technical differences in implementing the proof for the occupied and the vacant sets, so we write $\mathcal S$ to denote either of them in the discussion below. 
As in \cite{LS-Indistinguishability}, we argue by contradiction and assume that not all unbounded components of $\mathcal S$ have a same component property with positive probability, from which we conclude that there is a component property $\mathcal A$, such that with positive probability, $C_\mathcal S(0)$ is unbounded, has property $\mathcal A$, and there is a \emph{pivotal ball} for the event that $C_\mathcal S(0)$ has property $\mathcal A$,
that is, the occurrence of the event is sensitive to certain local modifications to the driving point process in the pivotal ball (see Section~\ref{sec:pivotal}).

Next, for an independent Poisson point process $\mathcal Y$, we declare $y\in\mathcal Y$ a $r$-trifurcation if removal of the connected component of $y$ in $\mathsf B(y,r)$ splits $C_\mathcal S(y)$ into at least $3$ unbounded components (Definition~\ref{def:r-trifurcation}). Each of these unbounded components contains $r$-trifurcations (Lemma~\ref{l:tree-of-trifurcations}), which allows us to define a \emph{random forest $F_\mathcal S$ on $r$-trifurcations} by connecting each $r$-trifurcation $y$ by an edge to its unique neighboring $r$-trifurcation in each of the unbounded branches of $C_\mathcal S(y)$ (see Section~\ref{sec:forest}). We then consider an independent reversible simple random walk $w$ on $F_\mathcal S$ and show that the environment viewed by the random walker is stationary (see Section~\ref{sec:rw}). 
Together with the first observation, it implies that the probability of the event $\{C_\mathcal S(w(n))$ is unbounded, has property $\mathcal A$, and there is a pivotal ball $B_n$  within distance $R$ from $w(n)\}$ is positive and does not depend on $n$. Since $C_\mathcal S(w(n)) = C_\mathcal S(0)$, the event $\{C_\mathcal S(0)$ is unbounded, has property $\mathcal A$, and there is a pivotal ball $B_n$ within distance $R$ from $w(n)\}$ has uniformly positive probability. 

Now, the event $\{C_\mathcal S(0)$ is unbounded and has property $\mathcal A\}$ can be approximated by a \emph{local event} with arbitrary precision. 
At the same time, the random walk $w$ is transient, so the pivotal ball $B_n$ will be arbitrarily far from the origin for all large $n$. This observation alone is not enough to get a contradiction though (cf.\ \cite[p.\ 1816]{LS-Indistinguishability}), since one should be able to make a pivotal modification of the driving point process inside the ball $B_n$. In other words, a \emph{pivotal modification} in $B_n$ should exist, which does not intervene in the history of the random walk $w$ up to time $n$. 
In \cite{LS-Indistinguishability}, the counterpart of $w$ was a nearest-neighbor random walk on the percolation cluster of the origin, so it was enough to condition that the random walk does not visit the pivotal edge before time $n$. In our case, the situation is more subtle, since already the definition of $r$-trifurcations depends on the global topology of the component. 
We succeed by imposing finer constraints on the location of the pivotal ball and the behavior of $w$ up to time $n$, which ensure that, after the modification, all the $r$-trifurcations visited by $w$ up to time $n$ together with their neighborhoods in $F_\mathcal S$ do not change (see events $\mathcal G_n$ in the proof of Theorem~\ref{thm:Indistinguishability}). 

To summarize, a key novel ingredient in our proof, compared to \cite{LS-Indistinguishability}, is to consider the random walk $w$ on $F_\mathcal S$.
The counterpart of $w$ in \cite{LS-Indistinguishability} is a nearest-neighbor random walk on the cluster of the origin and one first had to prove that it is transient (cf.\ \cite[Sections~8.3 and 8.6]{LyonsPeresBook}). By considering $w$ instead, we get a transient random walk for free; in return, some more care is needed to implement the local modification. 
Our argument can be applied in the setting of \cite{LS-Indistinguishability} and gives a shortcut to their proof, see Remark~\ref{rem:indistinguishability-graphs}.

\smallskip

One of the key tools for this paper and the main reason to confine ourselves to symmetric spaces is the \emph{mass-transport principle}, introduced to the study of percolation by H\"aggstr\"om \cite{Haggstrom97} and developed in continuum setting by Benjamini and Schramm \cite{BenjaminiSchramm-MTP}.
A function $\phi:\mathscr B(\mathsf X)\times \mathscr B(\mathsf X)\to \R_+$ is diagonally invariant if for all $B,B'\in\mathscr B(\mathsf X)$ and all isometries $\gamma$ of $\mathsf X$, $\phi(\gamma B, \gamma B') = \phi(B,B')$. 
\begin{lemma}[Mass-transport principle]\label{l:mtp}
Let $\phi:\mathscr B(\mathsf X)\times \mathscr B(\mathsf X)\to \R_+$ be diagonally invariant. If $\phi(B_0,\mathsf X)<\infty$ for some open set $B_0$, then 
\[
\phi(B,\mathsf X) = \phi(\mathsf X,B),\quad B\in\mathscr B(\mathsf X).
\]
\end{lemma}
The proof of Lemma~\ref{l:mtp} is given in \cite{BenjaminiSchramm-MTP} for hyperbolic plane, see \cite[Theorem~5.2]{BenjaminiSchramm-MTP}, but, as remarked there, it holds also for symmetric spaces. 

\smallskip

Let us now describe the structure of the paper. 
Section~\ref{sec:notation} contains some definitions and notation used throughout the paper.
Prelimiary results are contained in Sections~\ref{sec:ins-del-properties}--\ref{sec:pivotal}. 
In Section~\ref{sec:ins-del-properties}, we prove some basic properties of insertion- and deletion-tolerant point processes. 
In Section~\ref{sec:forest}, we define $r$-trifurcations and construct a forest on $r$-trifurcations, which captures some branching structure of connected components. 
In Section~\ref{sec:first-properties}, we prove some basic results about the number and structure of unbounded connected components in the occupied and vacant sets. Section~\ref{sec:rw} is devoted to random walks on random graphs with vertices in $\mathsf X$. The main result of the section is Proposition~\ref{prop:rw-stationarity}, in which a form of stationarity of environment viewed by the walker is proven. In Section~\ref{sec:pivotal}, we introduce the notion of a pivotal set and prove the existence of pivotal balls in the presence of unbounded components of opposite types, see Lemma~\ref{l:pivotals}. Section~\ref{sec:proof-indistinguishability} is devoted to the proofs of Theorem~\ref{thm:Indistinguishability} and Proposition~\ref{prop:Indistinguishability-hyperbolic}. We also discuss there some minor extensions of Theorem~\ref{thm:Indistinguishability}, which are needed to derive Theorems~\ref{thm:monotonicity-of-uniqueness} and \ref{thm:connectivity}, see Remark~\ref{rem:Indistinguishability-extension}. Theorem~\ref{thm:monotonicity-of-uniqueness} is proven in Section~\ref{sec:uniqueness-monotonicity} and Theorem~\ref{thm:connectivity} in Section~\ref{sec:connectivity}. In Section~\ref{sec:percolation}, we consider an independent Poisson-Boolean model on unbounded connected components of a random subset of $\mathsf X$ and prove the existence of percolation phase transition in each of the unbounded components, see Theorem~\ref{thm:percolation}. We also show there how Theorem~\ref{thm:percolation} implies Theorem~\ref{thm:percolation-intro}. In Section~\ref{sec:transience}, we consider the graph induced by a homogeneous Boolean model and prove that every infinite component of the graph, which contains trifurcations, is transient, see Theorem~\ref{thm:rw-transience}. We also show there how Theorem~\ref{thm:rw-transience} implies Theorem~\ref{thm:rw-transience-intro}. 
In Section~\ref{sec:random-radii}, we prove analogues of our results for general Boolean models with random i.i.d.\ radii driven by insertion- or deletion-tolerant processes, see Theorem~\ref{thm:main-iid}.

\section{Definitions and notation}\label{sec:notation}

In this section, we collect some definitions and notation used throughout the paper. 

\smallskip

We denote by $\mathsf X$ a non-compact symmetric space with metric $d_\mathsf X$, volume measure $\mu_\mathsf X$, and a marked origin $0\in\mathsf X$. We write $\mathsf B(x,r)$ for the closed ball of radius $r$ around $x\in\mathsf X$ and note that $\mu_\mathsf X(\mathsf B(x,r)) = \mu_\mathsf X(\mathsf B(0,r))$.

\smallskip

The space of simple counting measures on $\mathsf X$ is denoted by $\mathsf M(\mathsf X)$,  
\[
\mathsf M(\mathsf X) = \Big\{\omega=\sum\limits_{i\geq 1}\delta_{x_i}
\begin{array}{c}
\text{ with distinct }x_i\in \mathsf X\text{ and }\omega(B)<\infty\text{ for all }\\
B\in\mathscr B(\mathsf X)\text{ with }\mu_\mathsf X(B)<\infty
\end{array}\Big\}.
\]
Let $\mathscr M(\mathsf X)$ be the $\sigma$-algebra on $\mathsf M(\mathsf X)$ generated by the evaluation maps $\omega\mapsto\omega(B)$, for $B\in\mathscr B(\mathsf X)$. 
To make a distinction, we will call random elements of $\mathsf M(\mathsf X)$ \emph{point measures} on $\mathsf X$ and the support of random elements of $\mathsf M(\mathsf X)$ \emph{point processes} on $\mathsf X$. 
We denote the restriction of measure $\omega$ to $B\in \mathscr B(\mathsf X)$ by $\omega|_B$.

\smallskip

We denote by $\mathscr B_0(\mathsf X)$ the set of all sets $B\in\mathscr B(\mathsf X)$ with $\mu_\mathsf X(B)\in(0,\infty)$. 

For $B\in\mathscr B_0(\mathsf X)$, the measure $\mu_\mathsf X(\cdot\cap B)/\mu_\mathsf X(B)$ is called the \emph{uniform distribution} on $B$. 

\smallskip

For a probability measure $\mathbf P$ on $(\mathsf M(\mathsf X),\mathscr M(\mathsf X))$ and $B\in\mathscr B_0(\mathsf X)$, we denote by $\mathbf P^B$ the law of $\omega+\delta_X$, where $\omega$ has law $\mathbf P$ and $X$ is independent from $\omega$ and uniformly distributed in $B$ and by $\mathbf P_B$ the law of $\omega|_B$, where $\omega$ has law $\mathbf P$. 
In particular, $\mathbf P$ is \emph{insertion- (resp.\ deletion-)tolerant}, if $\mathbf P^B$ (resp.\ $\mathbf P|_{B^c}$) is absolutely continuous with respect to $\mathbf P$ for all $B\in\mathscr B_0(\mathsf X)$. 

\smallskip

We denote by $\mathcal S$ a random open or closed subset of $\mathsf X$. For $x\in\mathsf X$ and $r>0$, we write $C_{\mathcal S}(x)$ for the connected component of $x$ in $\mathcal S$ and 
$C_{\mathcal S}(x,r)$ for the connected component of $x$ in $\mathcal S\cap\mathsf B(x,r)$.
We denote by $N_\mathcal S$ the number of unbounded components in $\mathcal S$.

\smallskip

Each $\omega\in \mathsf M(\mathsf X)$ induces the closed subset $\mathcal O$ of $\mathsf X$, the \emph{occupied set}, defined by
\[
\mathcal O = \mathcal O(\omega) = \bigcup\limits_{x\in\mathrm{supp}(\omega)}\mathsf B(x,1)
\]
and its complement $\mathcal V = \mathsf X\setminus \mathcal O$, the \emph{vacant set}. 
A set $\mathcal A\in \mathscr B(\mathsf X)\otimes \mathscr M(\mathsf X)$ is called \emph{occupied component property} if $(x,\omega)\in\mathcal A$ implies that $(x',\omega)\in\mathcal A$ for all $x'$ connected to $x$ in $\mathcal O(\omega)$ 
and it is called \emph{vacant component property} if $(x,\omega)\in\mathcal A$ implies that $(x',\omega)\in\mathcal A$ for all $x'$ connected to $x$ in $\mathcal V(\omega)$. 
Given an occupied (resp.\ vacant) component property $\mathcal A$, we say that occupied (resp.\ vacant) component $\mathcal C$ has \emph{type} $\mathcal A$ if $(x,\omega)\in\mathcal A$ whenever $x\in\mathcal C$; otherwise (if $(x,\omega)\notin\mathcal A$ for all $x\in\mathcal C$), we say that $\mathcal C$ has type $\neg\mathcal A$. 
A component property $\mathcal A$ is isometry invariant, if $(x,\omega)\in\mathcal A$ implies that $(\gamma x,\gamma\omega)\in\mathcal A$ for all isometries $\gamma$ of $\mathsf X$.

\section{Properties of insertion- and deletion-tolerant measures}\label{sec:ins-del-properties}

In this section, we collect some properties of insertion- and deletion-tolerant measures, needed in the proofs; we refer to \cite{HolroydSoo} for some other properties. 

\begin{lemma}\label{l:conditional-probability-positive}
Let $B\in \mathscr B_0(\mathsf X)$. 
\begin{itemize}\itemsep4pt
\item[(a)]
If $\mathbf P$ is insertion-tolerant, then for every $\sigma(\omega|_{B^c})$-measurable $S\subseteq B$ with $\mu_\mathsf X(S)>0$ a.s.,
$\mathbf P[\omega(S)=\omega(B)=1\,|\,\omega|_{B^c}] >0$, $\mathbf P$-a.s. on the event $\omega(B)=0$.
\item[(b)]
If $\mathbf P$ is deletion-tolerant, then 
$\mathbf P[\omega(B)=0\,|\,\omega|_{B^c}] >0$, $\mathbf P$-a.s.
\end{itemize}
\end{lemma}
\begin{proof}
Assume there exists $A$ such that $\mathbf P[\omega|_{B^c}\in A, w(B)=0]>0$ and $\mathbf P[\omega(S)=w(B)=1,\omega|_{B^c}\in A]=0$, where $S$ is as in (a). However, 
\begin{eqnarray*}
\mathbf P^B[\omega(S)=\omega(B)=1,\omega|_{B^c}\in A] 
&= &\mathsf P\big[(\omega+\delta_X)(S)=(\omega+\delta_X)(B)=1, \omega|_{B^c}\in A\big]\\
&= &\mathsf P\big[X\in S, \omega(B)=0, \omega|_{B^c}\in A\big]\\
&= &\mathbf E\big[\tfrac{\mu_\mathsf X(S)}{\mu_\mathsf X(B)}\,\mathds{1}_{\{\omega(B)=0,\omega|_{B^c}\in A\}}\big] >0,
\end{eqnarray*}
which contradicts the absolute continuity of $\mathbf P^B$ with respect to $\mathbf P$. This proves (a).

\smallskip

Assume now there exists $A$ such that $\mathbf P[\omega|_{B^c}\in A]>0$ and $\mathbf P[\omega(B)=0,\omega|_{B^c}\in A]=0$. We have 
\[
\mathbf P_{B^c}[\omega(B)=0,\omega|_{B^c}\in A] 
= \mathbf P[\omega|_{B^c}(B)=0,\omega|_{B^c}\in A] 
=\mathbf P[\omega|_{B^c}\in A]>0,
\]
which contradicts the absolute continuity of $\mathbf P_{B^c}$ with respect to $\mathbf P$. This proves (b).
\end{proof}

\begin{lemma}\label{l:insertion-tolerance-for-conditional}
Let $\mathbf P$ be an isometry invariant probability measure on $(\mathsf M(\mathsf X),\mathscr M(\mathsf X))$. 
Let $E\in\mathscr M(\mathsf X)$ be an isometry invariant event with $\mathbf P[E]>0$.
\begin{itemize}
\item[(a)]
If $\mathbf P$ is insertion-tolerant, then conditional measure $\mathbf P[\cdot\,|\,E]$ is insertion-tolerant. 
\item[(b)]
If $\mathbf P$ is deletion-tolerant, then conditional measure $\mathbf P[\cdot\,|\,E]$ is deletion-tolerant. 
\end{itemize}
\end{lemma}
\begin{proof}
The proof is similar to the proof of \cite[Lemma~3.6]{LS-Indistinguishability}. We only show (a), since the proof of (b) is essentially the same. 

Firstly, let $E$ be a tail event of positive probability. Note that for all $x\in\mathsf X$, $\omega\in E$ iff $\omega+\delta_x\in E$. 

Let $\omega$ have distribution $\mathbf P$. For $B\in\mathscr B_0(\mathsf X)$, let $X$ be uniformly distributed in $B$ and independent from $\omega$. For every  $A\in\mathscr M(\mathsf X)$ with $\mathbf P[A\cap E]=0$, we have 
\[
\mathsf P[\omega+\delta_X\in A,\,\omega\in E] = \mathsf P[\omega+\delta_X\in A,\omega+\delta_X\in E] = \mathbf P^B[A\cap E] = 0,
\]
since $\mathbf P$ is insertion-tolerant. Thus, the law of $\omega+\delta_X$, given $\omega \in E$, is absolutely continuous with respect to $\mathbf P[\cdot\,|\,E]$; hence $\mathbf P[\cdot\,|\,E]$ is insertion-tolerant. 

\smallskip

Now, let $E\in\mathscr M(\mathsf X)$ be isometry invariant. By the isometry invariance of $E$ and $\mathbf P$, there exist events $E_n$ such that (a) $E_n$ is $\sigma\big(\omega|_{\mathsf B(0,n)^c}\big)$-measurable and (b) $\mathbf P[E\Delta E_n]<\frac1{2^n}$. Let $E' = \limsup E_n$. Note that $E'$ is a tail event and $\mathbf P[E\Delta E'] = 0$. Hence $\mathbf P[\cdot\,|\,E] = \mathbf P[\cdot\,|\,E']$ is insertion-tolerant. The proof is completed. 
\end{proof}

\section{Forest of trifurcations}\label{sec:forest}

In this section, we introduce trifurcations of unbounded connected components of a random set $\mathcal S$, marked by points of an auxiliary point process $\mathcal Y$, and define a random forest $F_\mathcal S$ on trifurcations in an isometry invariant way, which will capture some relevant information about branching structure of unbounded components. A reversible nearest-neighbor random walk on $F_\mathcal S$ will be instrumental in the proof of Theorem~\ref{thm:Indistinguishability}.

\smallskip

Let $\mathcal S$ be a random open or closed subset of $\mathsf X$. 
Let $\mathcal Y$ be a simple point process on $\mathsf X$ with an isometry invariant law and independent from $\mathcal S$.

\begin{definition}\label{def:r-trifurcation}
Let $r>0$. 
Point $y\in\mathcal Y$ is called \emph{$r$-trifurcation} for $\mathcal S$ if $C_{\mathcal S}(y)\setminus C_{\mathcal S}(y,r)$ contains at least $3$ unbounded connected components, $\mathsf B(y,1)\subset\mathcal S$, and there are no other points of $\mathcal Y$ within distance $2r$ from $y$. 
\end{definition}
The concept of trifurcations was originally introduced by Burton and Keane in \cite{BurtonKeane}. 
A version of the following result is well-known in the setting of percolation on graphs, see e.g.\ \cite[Proposition~8.33]{LyonsPeresBook}; in continuum, additional care is needed because of thin components. 
\begin{lemma}\label{l:tree-of-trifurcations}
Assume that the law of $\mathcal S$ is invariant under isometries of $\mathsf X$ and the expected number of connected components in $\mathcal S\cap\mathsf B(0,1)$ is finite. Let $r>0$. 
Almost surely if $y\in\mathcal Y$ is $r$-trifurcation for $\mathcal S$, then every unbounded component of $C_{\mathcal S}(y)\setminus C_{\mathcal S}(y,r)$ contains infinitely many $r$-trifurcations. 
\end{lemma}
\begin{proof}
Assume that with positive probability there is $r$-trifurcation $y\in\mathcal Y$ such that at least one of the unbounded components of $C_{\mathcal S}(y)\setminus C_{\mathcal S}(y,r)$ contains finitely many $r$-trifurcations. By isometry invariance, with positive probability, there is such trifurcation in $\mathsf B(0,1)$. 

\smallskip

Let $\mathcal Z$ be an independent Poisson point process on $\mathsf X$ with intensity $\mu_\mathsf X$. For $z\in\mathcal Z$, let $\mathcal {CC}(z)$ be the set of connected components of $\mathcal S\cap \mathsf B(z,1)$. 
For $k\in\mathbb N$ and $\mathcal C\in \mathcal {CC}(z)$, let $T_r(\mathcal C,k)$
be the set of all $r$-trifurcations $y$ in the connected component of $\mathcal C$ in $\mathcal S$ whose local component $C_{\mathcal S}(y,r)$ is at most $k$-th nearest to $\mathcal C$ with respect to the distance in $\mathcal S$ 
and consider the (random) function $m:\mathcal Z\times\mathcal Y\to[0,1]$ such that 
\[
m(z,y) = \sum\limits_{\mathcal C\in\mathcal{CC}(z)}\frac{1}{|T_r(\mathcal C,k)|}\mathds{1}_{\{y\in T_r(\mathcal C,k)\}}.
\]
In words, any $z\in\mathcal Z$ distributes a unit of mass uniformly among $r$-trifurcations from $T_r(\mathcal C,k)$ for each $\mathcal C\in\mathcal{CC}(z)$. Note that for each $k$, 
\[
\sum_{y\in \mathcal Y} m(z,y)\leq |\mathcal {CC}(z)|, \quad z\in \mathcal Z,
\]
and there exists $k<\infty$ such that 
\[
\mathsf P\Big[\text{there exists }y\in \mathcal Y\cap\mathsf B(0,1)\text{ such that }\sum_{z\in \mathcal Z} m(z,y) = \infty\Big] >0.
\]
Thus, the function $\phi:\mathscr B(\mathsf X)\times\mathscr B(\mathsf X)\to [0,1]$ defined as 
\[
\phi(A,B) = \mathsf E\Big[\sum\limits_{z\in\mathcal Z\cap A}\sum\limits_{y\in \mathcal Y\cap B}m(z,y)\Big] 
\]
is diagonally invariant and by the Fubini theorem and the isometry invariance of $\mathcal S$, 
\[
\phi(\mathsf B(0,1),\mathsf X) \leq \mathsf E\Big[\sum\limits_{z\in\mathcal Z\cap\mathsf B(0,1)} |\mathcal{CC}(z)|\Big] = \mu_\mathsf X\big(\mathsf B(0,1)\big)\mathsf E[N]<\infty, 
\]
where $N$ is the number of connected components in $\mathcal S\cap\mathsf B(0,1)$. But 
$\phi(\mathsf X,\mathsf B(0,1)) = \infty$, which contradicts the mass-transport principle (Lemma~\ref{l:mtp}). 
The proof is completed.
\end{proof}

\smallskip

Next, we define a graph $F_{\mathcal S}=(V_{\mathcal S},E_{\mathcal S})$ on $r$-trifurcations, similarly to \cite[p.\ 1352]{BLPS-critical}. Consider marked point process $\mathcal Y'$ by assigning independent $[0,1]$-uniform random labels to the points of $\mathcal Y$, cf.\ e.g.\ \cite[Section~5.2]{LastPenrose}. 
Vertex set $V_{\mathcal S}$ is the set of all $r$-trifurcations for $\mathcal S$. 
The set of oriented edges $\vec E_{\mathcal S}$ consists of all oriented pairs of $r$-trifurcations $(y,y')$ such that  $C_{\mathcal S}(y',r)$ is closest to $C_{\mathcal S}(y,r)$ with respect to the distance in $\mathcal S$ among all $r$-trifurcations $y''$ in the same connected component of $C_{\mathcal S}(y)\setminus C_{\mathcal S}(y,r)$ as $y'$, and, in case of multiple choices, $y'$ is the unique one with the smallest label. 
By Lemma~\ref{l:tree-of-trifurcations},
for every $r$-trifurcation $y$ and every unbounded connected component $\mathcal C$ of $C_{\mathcal S}(y)\setminus C_{\mathcal S}(y,r)$ there exists precisely one $r$-trifurcation $y'\in\mathcal C$ such that $(y,y')\in\vec E_{\mathcal S}$.
The set of unoriented edges $E_{\mathcal S}$ consists of all unoriented pairs of $r$-trifurcations $y$ and $y'$ such that either $(y,y')$ or $(y',y)$ is in $\vec E_{\mathcal S}$. 
\begin{lemma}\label{l:graph-forest}
Under the assumptions of Lemma~\ref{l:tree-of-trifurcations}, almost surely, $F_{\mathcal S}$ is either a forest with only infinite trees or empty.
\end{lemma}
\begin{proof}
By Lemma~\ref{l:tree-of-trifurcations}, if $F_{\mathcal S}$ is non-empty, then all its connected components are infinite. 
Assume that there is a cycle $(v_1,\ldots, v_k)$ in $F_{\mathcal S}$ whose vertices are all different. 
It follows from the definition of $F_{\mathcal S}$ that for every $i$, all the vertices $v_j$, $j\neq i$, are contained in a same unbounded connected component of $C_{\mathcal S}(v_i)\setminus C_{\mathcal S}(v_i,r)$. 
Thus, for every $i$, exactly one of the edges $(v_i,v_{i+1})$ or $(v_i,v_{i-1})$ is in $\vec E_{\mathcal S}$; here, $v_0=v_k$. 
Assume, without loss of generality, that $(v_0,v_1),\ldots, (v_{k-1},v_k)\in\vec E_{\mathcal S}$. 
Then for every $i$, the distance in $\mathcal S$ from the local connected component $C_{\mathcal S}(v_i,r)$ to $C_{\mathcal S}(v_{i+1},r)$ is not larger than the distance to $C_{\mathcal S}(v_{i-1},r)$. Thus, all the distances in $\mathcal S$ between $C_{\mathcal S}(v_i,r)$ and $C_{\mathcal S}(v_{i+1},r)$ must be equal. But then the label of $v_{i+1}$ is strictly smaller than the label of $v_{i-1}$ for all $i$, which is impossible. Hence $F_{\mathcal S}$ contains no cycles. 
\end{proof}

\begin{lemma}\label{l:graph-degrees}
Under the assumptions of Lemma~\ref{l:tree-of-trifurcations}, almost surely, every vertex in $F_{\mathcal S}$ has finite degree. 
\end{lemma}
\begin{proof}
Let $m:V_{\mathcal S}\times V_{\mathcal S}\to\{0,1\}$ be the adjacency matrix of $(V_{\mathcal S},\vec E_{\mathcal S})$. 
Note that the degree of $y$ in $F_{\mathcal S}$ is at most $\sum_{y'\in V_{\mathcal S}}m(y,y') + \sum_{y'\in V_{\mathcal S}}m(y',y)$. Consider the diagonally invariant function 
$\phi:\mathscr B(\mathsf X)\times\mathscr B(\mathsf X)\to [0,1]$ defined as 
\[
\phi(A,B) = \mathsf E\Big[\sum\limits_{y\in\mathcal Y\cap A}\sum\limits_{y'\in \mathcal Y\cap B}m(y,y')\Big].
\]
By the definition of $r$-trifurcation, there is at most one $r$-trifurcation $y$ in $\mathsf B(0,r)$ and the number of unbounded components in $C_\mathcal S(y)\setminus C_\mathcal S(y,r)$ is at most the number of components in all $\mathcal S\cap \mathsf B(x_i,1)$, $i\in I$, where $x_i\in \partial \mathsf B(0,2r+1)$ and $\cup_{i\in I}\mathsf B(x_i,1)\supseteq \partial \mathsf B(0,2r+1)$. Since the expectation of the latter is finite by the assumptions of the lemma and the definition of $F_\mathcal S$, we obtain that 
$\phi(\mathsf B(0,r),\mathsf X)<\infty$. Hence by the mass-transport principle (Lemma~\ref{l:mtp}), also $\phi(\mathsf X,\mathsf B(0,r))<\infty$. The former implies that almost surely $\sum_{y'\in V_{\mathcal S}}m(y,y')<\infty$, for all $y\in V_{\mathcal S}$, and the latter that $\sum_{y'\in V_{\mathcal S}}m(y',y)<\infty$. The proof is completed. 
\end{proof}

\begin{corollary}\label{cor:graph-transience}
Under the assumptions of Lemma~\ref{l:tree-of-trifurcations}, there exists $\rho>0$ such that almost surely, for every $x\in V_\mathcal S$, the probability that a simple random walk on $F_{\mathcal S}$ never returns to $x$ is bigger than $\rho$. 
\end{corollary}
\begin{proof}
By Lemma~\ref{l:graph-forest}, if $F_{\mathcal S}$ is non-empty, then it is a forest, whose trees are infinite and each vertex has degree at least $3$. The claim follows. 
\end{proof}

\section{First properties of the occupied and vacant sets}\label{sec:first-properties}

In this section, we prove basic properties on the number and structure of unbounded connected components in the occupied (resp.\ vacant) set of the Boolean model driven by insertion- (resp.\ deletion-)tolerant point measure, where we use a standard method of ``glueing'' of unbounded connected components.

\begin{lemma}\label{l:number-of-infinite-clusters-occupied}
Let $\mathbf P$ be isometry invariant. 
\begin{itemize}\itemsep4pt
\item[(a)]
If $\mathbf P$ is insertion-tolerant, then $N_\mathcal O\in \{0,1,\infty\}$ $\mathbf P$-almost surely.
\item[(b)]
If $\mathbf P$ is deletion-tolerant, then $N_\mathcal V\in \{0,1,\infty\}$ $\mathbf P$-almost surely.
\end{itemize}
\end{lemma}
\begin{proof}
The proof is quite standard. We begin with (a). Assume that $\mathbf P[N_\mathcal O=k]>0$ for some $2\leq k<\infty$. Since event $\{N_\mathcal O=k\}$ is isometry invariant, by Lemma~\ref{l:insertion-tolerance-for-conditional}, the measure $\mathbf P[\cdot\,|\,N_\mathcal O=k]$ is isometry invariant and insertion-tolerant. Hence, without loss of generality, we may assume that $\mathbf P[N_\mathcal O=k]=1$. 

Choose $n$ so that with positive probability $\mathsf B(0,n)$ intersects all $k$ unbounded components of $\mathcal O$. 
Let $B_1,\ldots, B_m$ be a covering of $\mathsf B(0,n)$ by balls of radius $\frac12$. Note that 
$\mathsf B(0,n)\subseteq \bigcup_{i=1}^m\mathsf B(x_i,1)$ for any $x_i\in B_i$. 
Let $X_1,\ldots, X_m$ be independent, independent from $\omega$, and such that $X_i$ is uniformly distributed in $B_i$. On the one hand, there is a unique unbounded component in 
$\mathcal O\cup \bigcup_{i=1}^m\mathsf B(X_i,1)$ with positive probability. On the other hand, since $\mathbf P$ is insertion-tolerant, the law of $\omega+\sum_{i=1}^m\delta_{X_i}$ is absolutely continuous with respect to $\mathbf P$; see \cite[Theorem~2]{HolroydSoo}.
Hence $\mathbf P[N_\mathcal O=1]>0$, which contradicts the assumption $\mathbf P[N_\mathcal O=k]=1$, so (a) is proven. 

\smallskip

To show (b), assume that $\mathbf P[N_\mathcal V=k]>0$ for some $2\leq k<\infty$. As in the proof of (a), we may assume without loss of generality that $\mathbf P[N_\mathcal V=k]=1$. Choose $n$ so that with positive probability $\mathsf B(0,n)$ intersects all $k$ unbounded components of $\mathcal V$. Then $\mathcal V(\omega|_{\mathsf B(0,n+1)^c})$ contains a unique unbounded component with positive probability. Hence, by the absolute continuity of $\mathbf P_{\mathsf B(0,n+1)^c}$ with respect to $\mathbf P$, $\mathbf P[N_\mathcal V=1]>0$, which contradicts the assumption $\mathbf P[N_\mathcal V=k]=1$, so (b) is proven. 
\end{proof}

\smallskip

Let $\mathcal Y$ be an independent Poisson point process on $\mathsf X$ with intensity $\mu_\mathsf X$. 
Recall Definition~\ref{def:r-trifurcation} of $r$-trifurcation. 

\begin{lemma}\label{l:infinite-clusters-trifurcations-2}
Let $\mathbf P$ be isometry invariant.
\begin{itemize}\itemsep4pt
\item[(a)]
If $\mathbf P$ is insertion-tolerant with $\mathbf P[N_\mathcal O=\infty]=1$, then almost surely, every unbounded connected component $\mathcal C$ of $\mathcal O$ contains a $r$-trifurcation for some $r=r(\mathcal C)$. 
\item[(b)]
If $\mathbf P$ is deletion-tolerant with $\mathbf P[N_\mathcal V=\infty]=1$ and finite expected number of connected components in $\mathcal V\cap\mathsf B(0,1)$, then almost surely, every unbounded connected component $\mathcal C$ of $\mathcal V$ contains a $r$-trifurcation for some $r=r(\mathcal C)$. 
\end{itemize}
\end{lemma}
\begin{proof}
We only show (a), since the proof of (b) is very similar. 
Assume that with positive probability there exists an unbounded component of $\mathcal O$ which contains no $r$-trifurcations for every $r$.
Choose $n$ so that with positive probability $\mathsf B(0,n)$ intersects at least $3$ unbounded connected components of $\mathcal O$, at least one of which contains no $r$-trifurcations for every $r$; denote this event by $E$. 
As in the proof of Lemma~\ref{l:number-of-infinite-clusters-occupied}, 
let $B_1,\ldots, B_m$ be a covering of $\mathsf B(0,n)$ by balls of radius $\frac12$ and let $X_1,\ldots, X_m$ be independent, independent from $\omega$ and $\mathcal Y$, and such that $X_i$ is uniformly distributed in $B_i$ and define $\mathcal O' = \mathcal O\cup \bigcup_{i=1}^m\mathsf B(X_i,1)$. 
Let $Y$ be uniformly distributed in $\mathsf B(0,1)$ and independent from everything else. Note that on event $E$, $Y$ satisfies the definition of $r$-trifurcation (Definition~\ref{def:r-trifurcation}) with 
$r=n+2$, $\mathcal S=\mathcal O'$ and $\mathcal Y$ replaced by $\mathcal Y'=\mathcal Y\big|_{\mathsf B(0,2n+5)^c}\cup\{Y\}$, and at least one of the unbounded connected components of $\mathcal C_{\mathcal O'}(Y)\setminus \mathcal C_{\mathcal O'}(Y,r)$ contains at most finitely many $r$-trifurcations. Since the law of $(\mathcal O',\mathcal Y')$ is absolutely continuous with respect to the law of $(\mathcal O,\mathcal Y)$, it follows that with positive probability for $r=n+2$, 
there exists a $r$-trifurcation $y\in\mathcal Y\cap \mathsf B(0,1)$ such that at least one of the infinite connected components of $C_\mathcal O(y)\setminus C_\mathcal O(y,r)$ contains only finitely many $r$-trifurcations. This contradicts Lemma~\ref{l:tree-of-trifurcations}. The proof is completed. 
\end{proof}

\section{Stationarity for random walks on random graphs}\label{sec:rw}

In this section, we consider reversible random walks on random graphs with vertices in $\mathsf X$ having isometry invariant law. Our aim is to show an instance of stationarity of the random environment viewed by the random walker, see Proposition~\ref{prop:rw-stationarity}; it is a key ingredient for the proof of Theorem~\ref{thm:Indistinguishability}. 

\smallskip

We identify oriented graphs with vertices in $\mathsf X$ with point measures in $\mathsf M(\mathsf X^2)$, so that the edge set of the graph is precisely the support of the point measure. Graph $G$ is unoriented if $(x,y)$ is in the support of $G$ iff $(y,x)$ is, and in what follows, we only consider unoriented graphs.
We write $V_G$ for the vertex set of $G$, $E_G$ for its edge set, $d_G(x)$ for the degree of vertex $x$ in $G$ and define $c_G(x) = d_G(x) + 1$. 
The transition probability of a simple delayed random walk on $G$ is defined by
\[
p_G(x,y) = \left\{\begin{array}{ll}\frac{1}{c_G(x)} & \text{if }\{x,y\}\in E_G\text{ or }x=y\in V_G\\ 0 &\text{else.}\end{array}\right.
\]
If $\gamma$ is an isometry of $\mathsf X$, we define $\gamma\circ G = \sum_{i\geq 1}\delta_{(\gamma (x_i),\gamma(y_i))}$, when $G = \sum_{i\geq 1}\delta_{(x_i,y_i)}$.

\smallskip

In this section, let $(\Xi,\Sigma, \mathbb P)$ be a probability space and $G:\Xi\to\mathsf M(\mathsf X^2)$ a measurable map whose images are unoriented graphs. We assume that isometries of $\mathsf X$ act on $\Xi$ so that they preserve $\mathbb P$ and $\gamma\circ (G(\xi)) = G(\gamma\xi)$ for all isometries $\gamma$ of $\mathsf X$. Furthermore, we assume that 
\begin{equation}\label{eq:graph-degree-condition}
Z := \mathbb E\Big[\sum\limits_{x_0\in V_G\cap\mathsf B(0,1)} c_G(x_0)\Big]<\infty.
\end{equation}

Given $x\in\mathsf X$ and $\xi\in\Xi$, we choose $w(0)$ at random among the vertices $V_G\cap\mathsf B(x,1)$ (if any), and consider two independent random walks $\{w(i)\}_{i\geq 0}$ and $\{w(-i)\}_{i\geq 0}$ from $w(0)$ with transition probabilities $p_G$. More precisely, we describe the joint law of $(w,\xi)$ by a unique probability measure $\mathbb {\widehat P}_x$ on $\mathsf X^\mathbb Z\times \Xi$, which satisfies  
for all $A\in\Sigma$, $k\in\mathbb N$ and $B_{-k},\ldots, B_k\in\mathscr B(\mathsf X)$, 
\begin{multline}\label{def:mathbb-widehat-Px}
\mathbb {\widehat P}_x\big[\big\{w(i)\in B_i,\,-k\leq i\leq k\big\}\times A\big]\\ 
= 
\frac1Z \,\mathbb E\Big[\mathds{1}_A(\xi)\,\Big(\sum\limits_{(x_{-k},\ldots,x_k)\in V^*} c_G(x_0)\,\prod\limits_{i=1}^k p_G(x_{i-1},x_i)\,\prod\limits_{i=-1}^{-k} p_G(x_{i+1},x_i)\Big)\Big],
\end{multline}
where $V^* = V_G^{2k+1}\cap \big(B_{-k}\times\ldots\times B_{-1}\times (\mathsf B(x,1)\cap B_0)\times B_1\times\ldots \times B_k\big)$.

\smallskip

The main result of this section is the following proposition, which is an adaptation 
to continuum setting of \cite[Lemma~3.13]{LS-Indistinguishability} (see also \cite{Haggstrom97, MofU-HaggstromPeres}). 

\begin{proposition}\label{prop:rw-stationarity}
Let $s:\mathsf X^{\mathbb Z}\to \mathsf X^{\mathbb Z}$ be the time shift defined by $(sw)(n) = w(n+1)$ and extend $s$ to $(\mathsf X^{\mathbb Z}, \Xi)$ by $s(w,\xi) = (sw,\xi)$. 
Then for every $\mathcal A\in \mathscr B(\mathsf X^{\mathbb Z})\otimes \Sigma$ invariant under isometries of $\mathsf X$ (acting diagonally on $\mathsf X^{\mathbb Z}\times \Xi$), 
\[
\widehat {\mathbb P}_0[s\mathcal A] = \widehat {\mathbb P}_0[\mathcal A].
\]
\end{proposition}

\smallskip

Consider the $\sigma$-finite measure $\mathbb Q$ on $\mathsf X^\mathbb Z\times \Xi$ defined by 
\[
\mathbb Q = \int\limits_{\mathsf X} \mu_\mathsf X(dx)\,\mathbb {\widehat P}_x.
\]
We begin by proving invariance properties of $\mathbb Q$.
If $\gamma$ is an isometry of $\mathsf X$, then we also denote by $\gamma$ the diagonal action of $\gamma$ on any power of $\mathsf X$ or on $\mathsf X^\mathbb Z\times \Xi$.

\begin{lemma}\label{l:Q-gamma-invariance}
Measure $\mathbb Q$ is invariant under isometries of $\mathsf X$. 
\end{lemma}
\begin{proof}
Let $\gamma$ be an isometry of $\mathsf X$. It suffices to prove that  
for all $A\in\Sigma$, $k\in\mathbb N$ and $B_{-k},\ldots, B_k\in\mathscr B(\mathsf X)$, 
\[
(\gamma\circ \mathbb Q)\big[\big\{w(i)\in B_i,\,-k\leq i\leq k\big\}\times A\big]
= \mathbb Q\big[\big\{w(i)\in B_i,\,-k\leq i\leq k\big\}\times A\big].
\]
We have 
\begin{multline*}
(\gamma\circ \mathbb Q)\big[\big\{w(i)\in B_i,\,-k\leq i\leq k\big\}\times A\big]
= \mathbb Q\big[\gamma^{-1}\big\{w(i)\in B_i,\,-k\leq i\leq k\big\}\times \gamma^{-1}A\big]\\
= \mathbb Q\big[\big\{w(i)\in \gamma^{-1}B_i,\,-k\leq i\leq k\big\}\times \gamma^{-1}A\big]\\
= \frac1Z\int\limits_{\mathsf X} \mu_\mathsf X(dx)\mathbb E\Big[\mathds{1}_{\gamma^{-1}A}(\xi)\,\Big(\sum\limits_{(x_{-k},\ldots,x_k)\in V'} c_G(x_0)\,\prod\limits_{i=1}^k p_G(x_{i-1},x_i)\,\prod\limits_{i=-1}^{-k} p_G(x_{i+1},x_i)\Big)\Big],
\end{multline*}
where $V' = V_G^{2k+1}\cap\big(\gamma^{-1}B_{-k}\times\ldots\times \gamma^{-1}B_{-1}\times (\mathsf B(x,1)\cap \gamma^{-1}B_0)\times \gamma^{-1}B_1\times\ldots \times \gamma^{-1}B_k\big)$. Since $c_G(\cdot) = c_{\gamma\circ G}(\gamma(\cdot))$, 
$p_G\big(\cdot,\cdot) = p_{\gamma\circ G}(\gamma(\cdot),\gamma(\cdot))$ 
and $\gamma\circ G(\xi) = G(\gamma\xi)$, we have 
\begin{multline*}
\mathbb E\Big[\mathds{1}_{\gamma^{-1}A}(\xi)\,\Big(\sum\limits_{(x_{-k},\ldots,x_k)\in V'} c_G(x_0)\,\prod\limits_{i=1}^k p_G(x_{i-1},x_i)\,\prod\limits_{i=-1}^{-k} p_G(x_{i+1},x_i)\Big)\Big]\\
= 
\mathbb E\Big[\mathds{1}_A(\gamma\xi)\,\Big(\sum\limits_{(x_{-k},\ldots,x_k)\in \gamma V'} c_{G(\gamma\xi)}(x_0)\,\prod\limits_{i=1}^k p_{G(\gamma\xi)}(x_{i-1},x_i)\,\prod\limits_{i=-1}^{-k} p_{G(\gamma\xi)}(x_{i+1},x_i)\Big)\Big].
\end{multline*}
Note that 
$\gamma V' = V_{G(\gamma\xi)}^{2k+1}\times(B_{-k}\times\ldots\times B_{-1}\times (\mathsf B(\gamma(x),1)\cap B_0)\times B_1\times\ldots \times B_k)$. 
Thus, by the isometry invariance of $\mu_\mathsf X$ and $\mathbb P$, 
\begin{multline*}
(\gamma\circ \mathbb Q)\big[\big\{w(i)\in B_i,\,-k\leq i\leq k\big\}\times A\big]\\
= 
\frac1Z\int\limits_{\mathsf X} \mu_\mathsf X(dx)\mathbb E\Big[\mathds{1}_A(\xi)\,\Big(\sum\limits_{(x_{-k},\ldots,x_k)\in V''} c_G(x_0)\,\prod\limits_{i=1}^k p_G(x_{i-1},x_i)\,\prod\limits_{i=-1}^{-k} p_G(x_{i+1},x_i)\Big)\Big],
\end{multline*}
where 
$V'' = V_G^{2k+1}\cap\big(B_{-k}\times\ldots\times B_{-1}\times (\mathsf B(x,1)\cap B_0)\times B_1\times\ldots \times B_k\big)$, 
which is equal to $\mathbb Q\big[\big\{w(i)\in B_i,\,-k\leq i\leq k\big\}\times A\big]$. The proof is completed. 
\end{proof}

\begin{lemma}\label{l:Q-s-invariance}
Measure $\mathbb Q$ is $s$-invariant. 
\end{lemma}
\begin{proof}
It suffices to prove that 
for all $A\in\Sigma$, $k\in\mathbb N$ and $B_{-k},\ldots, B_k\in\mathscr B(\mathsf X)$, 
\[
(s\circ \mathbb Q)\big[\big\{w(i)\in B_i,\,-k\leq i\leq k\big\}\times A\big]
= \mathbb Q\big[\big\{w(i)\in B_i,\,-k\leq i\leq k\big\}\times A\big].
\]
We have 
\begin{multline*}
(s\circ \mathbb Q)\big[\big\{w(i)\in B_i,\,-k\leq i\leq k\big\}\times A\big]
= \mathbb Q\big[s^{-1}\big\{w(i)\in B_i,\,-k\leq i\leq k\big\}\times A\big]\\
= \mathbb Q\big[\big\{w(i+1)\in B_i,\,-k\leq i\leq k\big\}\times A\big]
= \mathbb Q\big[\big\{w(i)\in B_{i-1},\,-k+1\leq i\leq k+1\big\}\times A\big]\\
= 
\frac1Z\int\limits_{\mathsf X} \mu_\mathsf X(dx)\mathbb E\Big[\mathds{1}_A(\xi)\,\Big(\sum\limits_{(x_{-k},\ldots,x_k)\in V'} c_G(x_0)\,\prod\limits_{i=1}^{k+1} p_G(x_{i-1},x_i)\,\prod\limits_{i=-1}^{-k+1} p_G(x_{i+1},x_i)\Big)\Big],
\end{multline*}
where $V'= V_G^{2k+1}\cap\big(B_{-k}\times\ldots\times B_{-2}\times (\mathsf B(x,1)\cap B_{-1})\times B_0\times\ldots \times B_k\big)$.
Note that 
\[
c_G(x_0)\,\prod\limits_{i=1}^{k+1} p_G(x_{i-1},x_i)\,\prod\limits_{i=-1}^{-k+1} p_G(x_{i+1},x_i)
=
c_G(x_1)\,\prod\limits_{i=2}^{k+1} p_G(x_{i-1},x_i)\,\prod\limits_{i=0}^{-k+1} p_G(x_{i+1},x_i).
\]
Therefore, by renaming each variable $x_{i+1}$ to $x_i$, the above integral is equal to 
\[
\frac1Z\int\limits_{\mathsf X} \mu_\mathsf X(dx)\mathbb E\Big[\mathds{1}_A(\xi)\,\Big(\sum\limits_{(x_{-k},\ldots,x_k)\in V'} c_G(x_0)\,\prod\limits_{i=1}^k p_G(x_{i-1},x_i)\,\prod\limits_{i=-1}^{-k} p_G(x_{i+1},x_i)\Big)\Big].
\]
By the Fubini theorem, this integral is equal to 
\[
\frac1Z\mathbb E\Big[\mathds{1}_A(\xi)\,\Big(\sum\limits_{(x_{-k},\ldots,x_k)\in V''} c_G(x_0)\,\prod\limits_{i=1}^k p_G(x_{i-1},x_i)\,\prod\limits_{i=-1}^{-k} p_G(x_{i+1},x_i)\,\mu_\mathsf X\big(\mathsf B(x_{-1},1)\big)\Big)\Big],
\]
where $V'' = B_{-k}\times\ldots\times B_{-2}\times B_{-1}\times B_0\times\ldots \times B_k$. Since $\mu_\mathsf X(\mathsf B(x_{-1},1)) = \mu_\mathsf X(\mathsf B(x_0,1))$, the expression equals 
\[
\frac1Z\int\limits_{\mathsf X} \mu_\mathsf X(dx)\mathbb E\Big[\mathds{1}_A(\xi)\,\Big(\sum\limits_{(x_{-k},\ldots,x_k)\in V'''} c_G(x_0)\,\prod\limits_{i=1}^k p_G(x_{i-1},x_i)\,\prod\limits_{i=-1}^{-k} p_G(x_{i+1},x_i)\Big)\Big],
\]
where $V'''=B_{-k}\times\ldots\times B_{-1}\times (\mathsf B(x,1)\cap B_0)\times B_1\times\ldots \times B_k$, which in turn is equal to $\mathbb Q\big[\big\{w(i)\in B_i,\,-k\leq i\leq k\big\}\times A\big]$. The proof is completed.
\end{proof}

\begin{lemma}\label{l:relation-Q-widehatPx}
For every $z\in\mathsf X$, 
\[
\widehat {\mathbb P}_z = \frac{1}{\mu_\mathsf X(\mathsf B(0,1))}\mathbb Q\big[\{w(0)\in \mathsf B(z,1)\}\cap\cdot\big].
\]
\end{lemma}
\begin{proof}
Let $z\in \mathsf X$. It suffices to prove that 
for $A\in\Sigma$, $k\in\mathbb N$ and $B_{-k},\ldots, B_k\in\mathscr B(\mathsf X)$, 
\begin{multline*}
\mathbb Q\big[\{w(0)\in \mathsf B(z,1)\}\cap \{w(i)\in B_i,\,-k\leq i\leq k\}\times A\big]\\
= \mu_\mathsf X\big(\mathsf B(0,1)\big)\,\widehat {\mathbb P}_z \big[\{w(i)\in B_i,\,-k\leq i\leq k\}\times A\big].
\end{multline*}
We have 
\begin{multline*}
\mathbb Q\big[\{w(0)\in \mathsf B(z,1)\}\cap \{w(i)\in B_i,\,-k\leq i\leq k\}\times A\big]\\
= 
\frac1Z\int\limits_{\mathsf X} \mu_\mathsf X(dx)\mathbb E\Big[\mathds{1}_A(\xi)\,\Big(\sum\limits_{(x_{-k},\ldots,x_k)\in V'} c_G(x_0)\,\prod\limits_{i=1}^k p_G(x_{i-1},x_i)\,\prod\limits_{i=-1}^{-k} p_G(x_{i+1},x_i)\Big)\Big],
\end{multline*}
where $V' = B_{-k}\times\ldots\times B_{-1}\times (\mathsf B(x,1)\cap B_0\cap \mathsf B(z,1))\times B_1\times\ldots \times B_k$. By the Fubini theorem, the iterated integral is equal to 
\[
\frac1Z\,\mathbb E\Big[\mathds{1}_A(\xi)\,\Big(\sum\limits_{(x_{-k},\ldots,x_k)\in V''} c_G(x_0)\,\prod\limits_{i=1}^k p_G(x_{i-1},x_i)\,\prod\limits_{i=-1}^{-k} p_G(x_{i+1},x_i)\,\mu_\mathsf X\big(\mathsf B(x_0,1)\big)\Big)\Big],
\]
where $V'' = B_{-k}\times\ldots\times B_{-1}\times (B_0\cap \mathsf B(z,1))\times B_1\times\ldots \times B_k$.
Since $\mu_\mathsf X(\mathsf B(x_0,1))=\mu_\mathsf X(\mathsf B(0,1))$,
the last expression equals $\mu_\mathsf X(\mathsf B(0,1))\,\widehat {\mathbb P}_z \big[\{w(i)\in B_i,\,-k\leq i\leq k\}\times A\big]$. The proof is completed. 
\end{proof}

\smallskip

We now have all the ingredients to prove Proposition~\ref{prop:rw-stationarity}. 
\begin{proof}[Proof of Proposition~\ref{prop:rw-stationarity}]
Fix a set $\mathcal A$ as in the statement and consider the function $\phi$ on $\mathscr B(\mathsf X)^2$ defined by 
\[
\phi(B,B') = \mathbb Q\big[\mathcal A\cap \{w(0)\in B\}\cap \{w(1)\in B'\}\big]. 
\]
If $\gamma$ is an isometry of $\mathsf X$, then
\begin{eqnarray*}
\phi(\gamma B,\gamma B') 
&= &\mathbb Q\big[\mathcal A\cap \{w(0)\in \gamma B\}\cap \{w(1)\in \gamma B'\}\big]\\
&= &\mathbb Q\big[\gamma \mathcal A\cap \gamma \{w(0)\in B\}\cap \gamma \{w(1)\in B'\}\big]\\
&= &\mathbb Q\big[\mathcal A\cap \{w(0)\in B\}\cap \{w(1)\in B'\}\big] = \phi(B,B'),
\end{eqnarray*}
where in the second equality we used invariance of $\mathcal A$ and in the third invariance of $\mathbb Q$.

Furthermore, by Lemma~\ref{l:relation-Q-widehatPx}, 
\[
\phi(\mathsf B(0,1),\mathsf X) = \mathbb Q\big[\mathcal A\cap \{w(0)\in \mathsf B(0,1)\}\big]
= \mu_\mathsf X(\mathsf B(0,1))\,\widehat {\mathbb P}_0[\mathcal A]<\infty. 
\]
Therefore, by the mass-transport principle, for every $B\in\mathscr B(\mathsf X)$, 
\begin{equation}\label{eq:Q-mtp}
\mathbb Q\big[\mathcal A\cap \{w(0)\in B\}\big] = \phi(B,\mathsf X) = \phi(\mathsf X,B)
= \mathbb Q\big[\mathcal A\cap \{w(1)\in B\}\big].
\end{equation}
Finally, we obtain by Lemma~\ref{l:relation-Q-widehatPx}, \eqref{eq:Q-mtp} and $s$-invariance of $\mathbb Q$ that  
\begin{eqnarray*}
\mu_\mathsf X(\mathsf B(0,1))\,\widehat {\mathbb P}_0[\mathcal A] 
&= & \mathbb Q\big[\mathcal A\cap \{w(0)\in \mathsf B(0,1)\}\big]
= \mathbb Q\big[\mathcal A\cap \{w(1)\in \mathsf B(0,1)\}\big]\\
&= &\mathbb Q\big[s\mathcal A\cap s\{w(1)\in \mathsf B(0,1)\}\big]
= \mathbb Q\big[s\mathcal A\cap \{w(0)\in \mathsf B(0,1)\}\big]\\
&= &\mu_\mathsf X(\mathsf B(0,1))\,\widehat {\mathbb P}_0[s\mathcal A].
\end{eqnarray*}
The proof is completed. 
\end{proof}

\section{Pivotal sets}\label{sec:pivotal}

The aim of this section is to introduce the concept of pivotal sets and to prove their exisence in the presence of unbounded connected components of opposite types. Recall the definitions of component property and type from Section~\ref{sec:notation}. 

\smallskip

Given an occupied component property $\mathcal A$, a set $B\in\mathscr B_0(\mathsf X)$ is called \emph{pivotal} for the occupied connected component of $x$ if $\omega(B)=0$ and there exists a measurable set $S=S(\omega)\subseteq B$ with $\mu_\mathsf X(S)>0$ a.s.\ such that precisely one of $(x,\omega)$ and $(x,\omega+\delta_y)$ is in $\mathcal A$, for every $y\in S$; in other words, for every $y\in S$, the connected component of $x$ in the enlarged occupied set $\mathcal O(\omega+\delta_y) = \mathcal O(\omega)\cup \mathsf B(y,1)$ has different type than the connected component of $x$ in $\mathcal O(\omega)$. 

Given a vacant component property $\mathcal A$, a set $B\in\mathscr B_0(\mathsf X)$ is called \emph{pivotal} for the vacant connected component of $x$ if precisely one of $(x,\omega)$ and $(x,\omega|_{B^c})$ is in $\mathcal A$; in other words, the connected component of $x$ in the enlarged vacant set $\mathcal V(\omega|_{B^c})$ has different type than the connected component of $x$ in $\mathcal V(\omega)$. 

\begin{lemma}[Existence of pivotals]\label{l:pivotals}
Let $\mathbf P$ be an isometry invariant probability measure on $\mathsf M(\mathsf X)$. 
Let $\mathcal Z$ be an independent Poisson point process on $\mathsf X$ with intensity $\mu_\mathsf X$. 
\begin{itemize}
\item[(a)]
Let $\mathcal A$ be an occupied component property. 
Assume that $\mathbf P$ is insertion-tolerant and with positive probability there exist unbounded occupied connected components of both types $\mathcal A$ and $\neg\mathcal A$. Then with positive probability $C_\mathcal O(0)$ is unbounded and $\mathsf B(z,\frac18)$ is pivotal for $C_\mathcal O(0)$ for some $z\in\mathcal Z$. 
\item[(b)]
Let $\mathcal A$ be a vacant component property. 
Assume that $\mathbf P$ is deletion-tolerant and with positive probability there exist unbounded vacant connected components of both types $\mathcal A$ and $\neg\mathcal A$. Then with positive probability $C_\mathcal V(0)$ is unbounded and $\mathsf B(z,2)$ is pivotal for $C_\mathcal V(0)$ for some $z\in\mathcal Z$. 
\end{itemize}
\end{lemma}

\begin{proof}
We begin with the proof of (a). 
For $s>0$, let $\mathcal E_s$ be the event that there exist unbounded occupied connected components of types $\mathcal A$ and $\neg\mathcal A$ at distance at most $s$ from each other, and define $s_0 = \inf\{s:\mathbf P[\mathcal E_s]>0\}$. Then there exist $a,b\in\mathsf X$ with $d_\mathsf X(a,b)\leq s_0+\frac14$ such that with positive probability, $C_\mathcal O(a)$ and $C_\mathcal O(b)$ are unbounded, $C_\mathcal O(a)$ has type $\mathcal A$ and $C_\mathcal O(b)$ has type $\neg\mathcal A$. We consider $\omega$'s from this event. 

\smallskip

Let $B = \mathsf B(a,\frac12)$. For every $x\in B$, $\mathsf B(x,1)\cap C_\mathcal O(a)\neq \emptyset$ and $d_\mathsf X\big(\mathsf B(x,1), C_\mathcal O(b)\big)\leq (s_0 - \frac14)_+$.  Since $C_\mathcal O(a)\neq C_\mathcal O(b)$ and by the definition of $s_0$, it follows that $\omega(B)=0$. 

\smallskip

Let $\omega^x = \omega+\delta_x$ and define $\mathcal O^x = \mathcal O(\omega^x) = \mathcal O(\omega)\cup\mathsf B(x,1)$. 
Note that for every $x\in B$, $d_\mathsf X\big(C_{\mathcal O^x}(a),C_{\mathcal O^x}(b)\big)\leq (s_0-\frac14)_+$. 
Assume that with positive $\mathbf P$-probability there is a measurable set $S'=S'(\omega)\subseteq B$ with $\mu_\mathsf X(S')>0$ such that for all $x\in S'$, $C_{\mathcal O^x}(a)$ has type $\mathcal A$ and $C_{\mathcal O^x}(b)$ has type $\neg\mathcal A$. 
Then with positive $\mathbf P^B$-probability, the occupied connected components of $a$ and $b$ are unbounded, have different types, and are at distance at most $(s_0-\frac14)_+$ from each other, which is impossible by the definition of $s_0$ and the absolute continuity of $\mathbf P^B$ with respect to $\mathbf P$. 
Thus, for $\mu_\mathsf X$-a.e.\ $x\in B$, either $C_{\mathcal O^x}(a)$ has different type than $C_\mathcal O(a)$ or $C_{\mathcal O^x}(b)$ has different type than $C_\mathcal O(b)$. 
It follows that either with positive $\mathbf P$-probability $\omega(B)=0$ and $\mathsf B(a,\frac1{16})$ is pivotal for $C_\mathcal O(a)$ or with positive $\mathbf P$-probability $\omega(B)=0$ and $\mathsf B(a,\frac1{16})$ is pivotal for $C_\mathcal O(b)$. 

\smallskip

All in all, by the isometry invariance of $\mathbf P$, there exists $x_*\in\mathsf X$ such that with positive $\mathbf P$-probability $C_\mathcal O(0)$ is unbounded, $\omega\big(\mathsf B(x_*,\frac12)\big) =0$ and $\mathsf B(x_*,\frac1{16})$ is pivotal for $C_\mathcal O(0)$. 
Note that for every $z\in \mathsf B(x_*,\frac1{16})$, $\mathsf B(x_*,\frac1{16})\subset\mathsf B(z,\frac18)\subset \mathsf B(x_*,\frac12)$. Hence with positive $\mathbf P$-probability, $C_\mathcal O(0)$ is unbounded and $\mathsf B(z,\frac18)$ is pivotal for $C_\mathcal O(0)$ for $\mu_\mathsf X$-a.e.\ $z\in \mathsf B(x_*,\frac1{16})$.
Since $\mathcal Z\cap \mathsf B(x_*,\frac1{16})\neq \emptyset$ with positive probability, result (a) follows. 

\smallskip

We proceed with the proof of (b). For $s>0$, let $\mathcal E_s$ be the event that there exist unbounded vacant connected components of types $\mathcal A$ and $\neg\mathcal A$ at distance at most $s$ from each other, and define $s_0 = \inf\{s:\mathbf P[\mathcal E_s]>0\}$. Then there exist $a,b\in\mathsf X$ with $d_\mathsf X(a,b)\leq s_0+\frac14$ such that with positive probability, $C_\mathcal V(a)$ and $C_\mathcal V(b)$ are unbounded, $C_\mathcal V(a)$ has type $\mathcal A$ and $C_\mathcal V(b)$ has type $\neg\mathcal A$. We consider $\omega$'s from this event. 

\smallskip

Let $\omega_x$ be the restriction of $\omega$ to $\mathsf X\setminus \mathsf B(x,2)$ and define $\mathcal V_x = \mathcal V(\omega_x)$. 
Note that inclusions $\mathsf B(a,\frac12)\subset\mathsf B(x,1)\subset \mathcal V_x$ hold for every $x\in \mathsf B(a,\frac12)$. 
Hence for every $x\in \mathsf B(a,\frac12)$, $d_\mathsf X\big(C_{\mathcal V_x}(a), C_{\mathcal V_x}(b)\big)\leq (s_0-\frac14)_+$. As in the proof of (a), we conclude from deletion tolerance of $\mathbf P$ and the definition of $s_0$ that for $\mu_\mathsf X$-a.e.\ $x\in\mathsf B(a,\frac12)$, either $C_{\mathcal V_x}(a)$ has different type than $C_\mathcal V(a)$ or $C_{\mathcal V_x}(b)$ has different type than $C_\mathcal V(b)$. It follows that either with positive $\mathbf P$-probability the set of $x\in\mathsf B(a,\frac12)$ for which $\mathsf B(x,2)$ is pivotal for $C_\mathcal V(a)$ has $\mu_\mathsf X$-measure at least $\frac12$ or with positive $\mathbf P$-probability the set of $x\in\mathsf B(a,\frac12)$ for which $\mathsf B(x,2)$ is pivotal for $C_\mathcal V(b)$ has $\mu_\mathsf X$-measure at least $\frac12$. 
All in all, by the isometry invariance of $\mathbf P$, there exists $x_*\in\mathsf X$ such that with positive $\mathbf P$-probability, $C_\mathcal V(0)$ is unbounded and the (random) set $S\subset \mathsf B(x_*,\frac12)$ of all $x$ for which $\mathsf B(x,2)$ is pivotal for $C_\mathcal V(0)$ has measure $\mu_\mathsf X(S)\geq \frac12$. 
Finally, on the latter event, Poisson point process $\mathcal Z$ intersects $S$ with positive probability, which proves (b). 
\end{proof}

\section{Proofs of Theorem~\ref{thm:Indistinguishability} and Proposition~\ref{prop:Indistinguishability-hyperbolic}}\label{sec:proof-indistinguishability}

In this section, we prove Theorem~\ref{thm:Indistinguishability} and Proposition~\ref{prop:Indistinguishability-hyperbolic}. 
We also discuss an extension of Theorem~\ref{thm:Indistinguishability} to marked point processes, which we use in Section~\ref{sec:uniqueness-monotonicity} to establish the uniqueness monotonicity (Theorem~\ref{thm:monotonicity-of-uniqueness}), see Remark~\ref{rem:Indistinguishability-extension}. 
The proof of Theorem~\ref{thm:Indistinguishability} is based on all the results of the previous sections. 
The proof of Proposition~\ref{prop:Indistinguishability-hyperbolic} is completely independent from the other results of the paper. 

\begin{proof}[Proof of Theorem~\ref{thm:Indistinguishability}]
We begin with the proof of (a). 
Assume that the statement is false, that is there exists an occupied component property $\mathcal A$ such that with positive probability there exist unbounded occupied components of both types $\mathcal A$ and $\neg \mathcal A$. By Lemma~\ref{l:number-of-infinite-clusters-occupied}, on this event the number of unbounded occupied components is infinite. Since by Lemma~\ref{l:insertion-tolerance-for-conditional}, $\mathbf P[\cdot\,|\,N_\mathcal O=\infty]$ satisfies the assumptions of the theorem, we may assume without loss of generality that the number of unbounded occupied components is infinite $\mathbf P$-almost surely. 

\smallskip

Let $\mathcal Z$ be a Poisson point processes on $\mathsf X$ with intensity $\mu_\mathsf X$ independent from $\omega$. 
By replacing $\mathcal A$ with $\neg\mathcal A$ if necessary, we may assume, by Lemma~\ref{l:pivotals}, that with positive probability $C_\mathcal O(0)$ is unbounded, has type $\mathcal A$, and $\mathsf B(z,\frac18)$ is pivotal (with respect to type $\mathcal A$) for $C_\mathcal O(0)$ for some $z\in\mathcal Z$. Denote by $\mathcal E_0$ the resulting event of positive probability.

\smallskip

Let $\mathcal Y$ be a Poisson point processes on $\mathsf X$ with intensity $\mu_\mathsf X$ independent from $\omega$ and $\mathcal Z$. By Lemmas~\ref{l:infinite-clusters-trifurcations-2} and \ref{l:tree-of-trifurcations}, there exists $r>0$ such that with positive probability $\mathcal E_0$ occurs and $C_\mathcal O(0)$ contains infinitely many $r$-trifurcations (in $\mathcal Y$). 
Note that if $y,y'\in\mathcal Y$ are $r$-trifurcation for $C_\mathcal O(0)$ and $\mathcal C_1,\ldots, \mathcal C_m$ resp.\ $\mathcal C_1',\ldots, \mathcal C_{m'}'$ are all the finitely many unbounded connected components of $C_\mathcal O(y)\setminus C_\mathcal O(y,r)$ resp.\ $C_\mathcal O(y')\setminus C_\mathcal O(y',r)$, then there exist $i$ and $j$ such that $\mathcal C_k'\subset \mathcal C_i$ for all $k\neq j$. Therefore, if some $\mathsf B(z,\frac18)$ is pivotal for $C_\mathcal O(0)$, then there exists $r$-trifurcation $y\in C_\mathcal O(0)$ such that (a) $\mathsf B(z,\frac98)$ intersects at most one connected component of $C_\mathcal O(y)\setminus C_\mathcal O(y,r)$ and (b) 
$C_\mathcal O(y)\setminus \big(C_\mathcal O(y,r)\cup\mathsf B(z,\frac98)\big)$ contains at least $3$ unbounded connected components that touch $C_\mathcal O(y,r)$.
Conditions (a) and (b) ensure that $y$ remains $r$-trifurcation even if $\omega$ is modified on $\mathsf B(z,\frac18)$, which will be crucially used later in the proof.
Together with the isometry invariance of $\mathbf P$ and the laws of $\mathcal Y$ and $\mathcal Z$, we obtain that for some $R>0$, with positive probability 
\begin{itemize}
\item
there exists $y\in\mathcal Y\cap\mathsf B(0,1)$ such that $C_\mathcal O(y)$ is unbounded, has type $\mathcal A$, and $y$ is a $r$-trifurcation for $C_\mathcal O(y)$,
\item
there exists $z\in\mathcal Z\cap \mathsf B(y,R)$ such that $\mathsf B(z,\frac18)$ is pivotal for $C_\mathcal O(y)$,
$\mathsf B(z,\frac98)$ intersects at most one connected component of $C_\mathcal O(y)\setminus C_\mathcal O(y,r)$ and 
$C_\mathcal O(y)\setminus \big(C_\mathcal O(y,r)\cup\mathsf B(z,\frac98)\big)$ contains at least $3$ unbounded connected components that touch $C_\mathcal O(y,r)$.
\end{itemize}
We fix such $R$ and denote by $\mathcal E_1$ the resulting event of positive probability. (Note that on event $\mathcal E_1$, $C_\mathcal O(y) = C_\mathcal O(0)$ by the definition of $r$-trifurcation.)

\smallskip

For $y,z\in\mathsf X$, denote by $\mathsf S(z;y)$ the set of all $x\in\mathsf B(z,\frac18)$ such that the 
connected component of $y$ in $\mathcal O(\omega|_{\mathsf B(z,\frac18)^c})$ and connected component of $y$ in $\mathcal O(\omega|_{\mathsf B(z,\frac18)^c}+\delta_x)$ have different types ($\mathcal A$ or $\neg \mathcal A$). Note that $\mathsf S(z;y)$ is $\sigma(\omega|_{\mathsf B(z,\frac18)^c})$-measurable and if $\mathsf B(z,\frac18)$ is pivotal for $C_\mathcal O(y)$, then $\mu_\mathsf X(\mathsf S(z;y))>0$ a.s.\ (here we use that $\omega(\mathsf B(z,\frac18)) =0$ to infer $\omega=\omega|_{\mathsf B(z,\frac18)^c}$). 
Let $\mathcal U(z;y)$ be the event that $\omega(\mathsf S(z;y)) = \omega(\mathsf B(z,\frac18)) = 1$. 
Then by Lemma~\ref{l:conditional-probability-positive}, there exists $\sigma>0$ such that with positive probability event $\mathcal E_1$ occurs and the point $z$ in the definition of $\mathcal E_1$ satisfies additionally that 
$\mathbf P\big[\mathcal U(z;y)\,\big|\,\omega|_{\mathsf B(z,\frac18)^c}\big]\geq \sigma$.\footnote{Note that the map $(z,\omega)\mapsto \omega|_{\mathsf B(z,\frac18)^c}$ is measurable, which follows, e.g., from the fact that the map $(z,\omega)\mapsto \gamma_z\omega$ is measurable, where $\gamma_z$ is an isometry of $\mathsf X$ with $\gamma_z(z) = 0$, c.f.\ \cite[Lemma~9.2]{LastPenrose}.} 
We fix such $\sigma$ and denote by $\mathcal E_2$ the resulting event of positive probability.

\smallskip

Consider a probability space $(\Xi,\Sigma,\mathbb P)$ on which the following independent processes are defined: (a) point measure $\omega$ with law $\mathbf P$, (b) Poisson point process $\mathcal Y$ equipped with independent $[0,1]$-uniform marks and (c) Poisson point process $\mathcal Z$.
Let $F_\mathcal O$ be the forest defined from the point processes (a) and (b) as in Section~\ref{sec:forest} and denote by $\widehat{\mathbb P}_0$ the joint law of random element $\xi\in\Xi$ and a doubly infinite lazy random walk $w$ on $F_\mathcal O$ as defined in \eqref{def:mathbb-widehat-Px}. 
For $n\in\Z$, consider the event $\mathcal G_n$ that 
\begin{itemize}
\item
$C_\mathcal O(w(n))$ is unbounded and has type $\mathcal A$;
\item
there exists $z\in\mathcal Z\cap \mathsf B(w(n),R)$ such that 
\begin{itemize}
\item
$\omega\big(\mathsf B(z,\frac18)\big)=0$ and 
$\mathbf P\big[\mathcal U(z;w(n))\,\big|\,\omega|_{\mathsf B(z,\frac18)^c}\big]\geq \sigma$,
\item
$\mathsf B(z,\frac98)$ intersects at most one connected component of $C_\mathcal O(w(n))\setminus C_\mathcal O(w(n),r)$ and 
$C_\mathcal O(w(n))\setminus \big(C_\mathcal O(w(n),r)\cup\mathsf B(z,\frac98)\big)$ contains at least $3$ unbounded connected components that touch $C_\mathcal O(w(n),r)$.
\end{itemize}
\item
the past of the random walk $\{w(n'), n'<n\}$ is contained in an unbounded connected component of $C_\mathcal O(w(n))\setminus C_\mathcal O(w(n),r)$ that does not intersect $\mathsf B(z,\frac98)$. 
\end{itemize}
Notice that $C_\mathcal O(w(0)) = C_\mathcal O(0)$ $\widehat{\mathbb P}_0$-almost surely. 
In particular, event $\mathcal E_2$ implies the first two conditions of $\mathcal G_0$ almost surely. 
Together with Corollary~\ref{cor:graph-transience}, these imply that $\widehat{\mathbb P}_0[\mathcal G_0] >0$. 
Furthermore, events $\mathcal G_n$ are invariant under diagonal actions of isometries of $\mathsf X$ and $s\mathcal G_n = \mathcal G_{n-1}$, where $s$ is the time shift defined in Proposition~\ref{prop:rw-stationarity}, hence, by Proposition~\ref{prop:rw-stationarity}, they have the same probability,  
\begin{equation}\label{eq:infGn}
\inf\limits_{n\in\Z}\mathbb {\widehat P}_0[\mathcal G_n] = 
\mathbb {\widehat P}_0[\mathcal G_0]>0.
\end{equation}
To complete the proof, we show that the above infimum is equal to $0$. 

\smallskip

Fix $\varepsilon>0$. 
Let $\mathcal A_0$ be the event that $C_\mathcal O(0)$ is unbounded and has type $\mathcal A$ and let $\mathcal A_0'$ be a local event such that $\mathbf P[\mathcal A_0\Delta \mathcal A_0']<\varepsilon$. 
Since for every $n\in\Z$, $C_\mathcal O(w(n))=C_\mathcal O(0)$ $\mathbb {\widehat P}_0$-almost surely, we have $\mathbb {\widehat P}_0[\mathcal G_n]= \mathbb {\widehat P}_0[\mathcal A_0\cap\mathcal G_n]
\leq \mathbb {\widehat P}_0[\mathcal A_0'\cap\mathcal G_n] + \varepsilon$. 
Fix $L$ such that $\mathcal A_0'$ depends only on the restriction of $\omega$ to $\mathsf B(0,L)$.
Further, let $\mathcal G_n' = \mathcal G_n\cap \{w(n)\notin \mathsf B(0, L+R+\frac18)\}$. 
Since $w$ is almost surely transient, for all $n$ large enough, 
$\mathbb {\widehat P}_0 [\mathcal G_n\setminus \mathcal G_n']<\varepsilon$. Hence for all large $n$, 
\[
\mathbb {\widehat P}_0[\mathcal G_n]
\leq \mathbb {\widehat P}_0[\mathcal A_0'\cap\mathcal G_n'] + 2\varepsilon.
\]

\smallskip

For $z\in\mathsf X$, consider event $\mathcal G_{n,z}'$ that 
\begin{itemize}\itemsep4pt
\item[(a)]
$C_\mathcal O(0)$ is unbounded and has type $\mathcal A$, 
\item[(b)]
$\omega\big(\mathsf B(z,\frac18)\big)=0$ and 
$\mathbf P \big[\mathcal U(z;0)\,\big|\,\omega|_{\mathsf B(z,\frac18)^c}\big]\geq \sigma$,
\item[(c)]
$w(n)\in\mathsf B(z,R)$ and $w(n)\notin \mathsf B(0, L+R+\frac18)$,
\item[(d)]
$\mathsf B(z,\frac98)$ intersects at most one connected component of $C_\mathcal O(w(n))\setminus C_\mathcal O(w(n),r)$ and 
$C_\mathcal O(w(n))\setminus \big(C_\mathcal O(w(n),r)\cup\mathsf B(z,\frac98)\big)$ contains at least $3$ unbounded connected components that touch $C_\mathcal O(w(n),r)$,
\item[(e)]
the past of the random walk $\{w(n'), n'<n\}$ is contained in an unbounded connected component of $C_\mathcal O(w(n))\setminus C_\mathcal O(w(n),r)$ that does not intersect $\mathsf B(z,\frac98)$.
\end{itemize}

\smallskip

Note that $\mathcal G_n' = \bigcup\limits_{z\in\mathcal Z} \mathcal G_{n,z}'$.
Let $\mathbf P'$ be the law of $\mathcal Z$. By the independence of $\mathcal Z$ and the Campbell formula (see e.g.\ \cite[Proposition~2.7]{LastPenrose}), we obtain that 
\begin{equation}\label{eq:Campbell-1}
\mathbb {\widehat P}_0[\mathcal A_0'\cap\mathcal G_n']
\leq \mathbf E'\Big[\sum\limits_{z\in\mathcal Z} 
\mathbb {\widehat P}_0[\mathcal A_0'\cap\mathcal G_{n,z}']\Big]
= \int\limits_{\mathsf X} \mu_\mathsf X(dz) \, \mathbb {\widehat P}_0[\mathcal A_0'\cap\mathcal G_{n,z}'].
\end{equation}

If $z\in\mathsf B(0,L+\frac18)$, then $\mathbf {\widehat P}_0[\mathcal A_0'\cap\mathcal G_{n,z}']=0$. 
For $z\notin\mathsf B(0,L+\frac18)$, define $\mathcal O_z = \mathcal O(\omega|_{\mathsf B(z,\frac18)^c})$ and consider event $\mathcal G_{n,z}''$ that 
\begin{itemize}\itemsep4pt
\item
$C_{\mathcal O_z}(0)$ is unbounded and has type $\mathcal A$,
\item
$\mathbf P \big[\mathcal U(z;0)\,\big|\,\omega|_{\mathsf B(z,\frac18)^c}\big]\geq \sigma$,
\item
(c)-(e) from the definition of $\mathcal G_{n,z}'$ hold. 
\end{itemize}
Note that $\mathcal G_{n,z}''$ does not depend on $\omega|_{\mathsf B(z,\frac18)}$. Indeed, if events (d) and (e) occur then for any modification of $\omega$ in $\mathsf B(z,\frac18)$, all $w(n')$, $n'<n$, remain $r$-trifurcations and their neighborhoods in $F_\mathcal O$ remain unchanged (although the neighborhood of $w(n)$ in $F_\mathcal O$ may change). Furthermore, since $\omega=\omega|_{\mathsf B(z,\frac18)^c}$ on $\mathcal G_{n,z}'$, we have inclusion $\mathcal G_{n,z}'\subset\mathcal G_{n,z}''$. Also, since $z\notin\mathsf B(0,L+\frac18)$, local event $\mathcal A_0'$ does not depend on $\omega|_{\mathsf B(z,\frac18)}$. Thus, 
\begin{eqnarray*}
\mathbb {\widehat P}_0[\mathcal A_0'\cap\mathcal G_{n,z}']
&\leq &\mathbb {\widehat P}_0\Big[\mathcal A_0', \mathcal G_{n,z}'', \mathbf P \big[\mathcal U(z;0)\,\big|\,\omega|_{\mathsf B(z,\frac18)^c}\big]\geq \sigma\Big]\\
&\leq &\frac1\sigma\,\mathbb {\widehat E}_0\Big[\mathds{1}_{\mathcal A_0'\cap \mathcal G_{n,z}''} \mathbf P \big[\mathcal U(z;0)\,\big|\,\omega|_{\mathsf B(z,\frac18)^c}\big]\Big]\\
&= &\frac1\sigma\,\mathbb {\widehat P}_0\big[\mathcal A_0', \mathcal G_{n,z}'',\mathcal U(z;0)\big].
\end{eqnarray*}
On event $\mathcal U(z;0)$, $C_{\mathcal O_z}(0)$ and $C_\mathcal O(0)$ have different types. Hence 
\[
\mathbb {\widehat P}_0\big[\mathcal A_0', \mathcal G_{n,z}'',\mathcal U(z;0)\big]
\leq 
\mathbb {\widehat P}_0\big[\mathcal A_0', \mathcal A_0^c, w(n) \in \mathsf B(z,R)\big].
\]
Thus, by the Fubini theorem and isometry invariance of $\mu_\mathsf X$,
\[
\mathbb {\widehat P}_0[\mathcal A_0'\cap\mathcal G_n']
\leq \frac1\sigma\,\int\limits_{\mathsf X} \mu_\mathsf X(dz)\,\mathbf {\widehat P}_0\big[\mathcal A_0', \mathcal A_0^c, w(n) \in \mathsf B(z,R)\big]
= \frac1\sigma\,\mu_\mathsf X\big(\mathsf B(0,R)\big)\,\mathbb {\widehat P}_0[\mathcal A_0', \mathcal A_0^c]\\
\]
All in all, 
\[
\limsup\limits_{n\to\infty} \mathbb {\widehat P}_0[\mathcal G_n]
\leq 2\varepsilon + \frac1\sigma\,\mu_\mathsf X\big(\mathsf B(0,R)\big)\,\mathbb {\widehat P}_0[\mathcal A_0', \mathcal A_0^c]
\leq 2\varepsilon + \frac1\sigma\,\mu_\mathsf X\big(\mathsf B(0,R)\big)\,\varepsilon.
\]
Since $\varepsilon>0$ is arbitrary, we arrive at contradiction with \eqref{eq:infGn}, so (a) is proven.

\smallskip

The proof of (b) is essentially the same as the proof of (a), so we only discuss the necessary changes.
Let $\omega$ be a point measure with law $\mathbf P$. 
By Lemma~\ref{l:insertion-tolerance-for-conditional}, without loss of generality, we may assume that the number of unbounded components in $\mathcal V(\omega)$ is infinite almost surely.
Let $\mathcal Z$ be a Poisson point processes on $\mathsf X$ with intensity $\mu_\mathsf X$ independent from $\omega$. 
Assuming the statement (b) is false, we find (by Lemma~\ref{l:pivotals}) a vacant component property $\mathcal A$ such that with positive probability, 
$C_\mathcal V(0)$ is unbounded, has type $\mathcal A$, and $\mathsf B(z,2)$ is pivotal (with respect to type $\mathcal A$) for $C_\mathcal V(0)$ for some $z\in\mathcal Z$; call this event $\mathcal E_0$. 

\smallskip

Let $\mathcal Y$ be a Poisson point processes on $\mathsf X$ with intensity $\mu_\mathsf X$ independent from $\omega$ and $\mathcal Z$. By Lemmas~\ref{l:infinite-clusters-trifurcations-2} and \ref{l:tree-of-trifurcations}, there exists $r>0$ such that with positive probability $\mathcal E_0$ occurs and $C_\mathcal V(0)$ contains infinitely many $r$-trifurcations (in $\mathcal Y$). 
Argueing as in (a), we observe that for some $R>0$, with positive probability, 
\begin{itemize}
\item
there exists $y\in\mathcal Y\cap\mathsf B(0,1)$ such that $C_\mathcal V(y)$ is unbounded, has type $\mathcal A$, and $y$ is a $r$-trifurcation for $C_\mathcal V(y)$, and
\item
there exists $z\in\mathcal Z\cap \mathsf B(y,R)$ such that $\mathsf B(z,2)$ is pivotal for $C_\mathcal V(y)$,
$\mathsf B(z,3)$ intersects at most one connected component of $C_\mathcal V(y)\setminus C_\mathcal V(y,r)$ and 
$C_\mathcal V(y)\setminus \big(C_\mathcal V(y,r)\cup\mathsf B(z,3)\big)$ contains at least $3$ unbounded connected components that touch $C_\mathcal V(y,r)$.
\end{itemize}
Call this event $\mathcal E_1$ and note that $C_\mathcal O(y) = C_\mathcal O(0)$ on $\mathcal E_1$, by the definition of $r$-trifurcation.

\smallskip

Finally, by Lemma~\ref{l:conditional-probability-positive}, there exists $\sigma>0$ such that with positive probability, $\mathcal E_1$ occurs and 
$\mathbf P\big[\omega(\mathsf B(z,2))=0\,\big|\,\omega|_{\mathsf B(z,2)^c}\big]\geq \sigma$, where $z$ is the same as in the definition of $\mathcal E_1$.

\smallskip

Consider a probability space $(\Xi,\Sigma,\mathbb P)$ on which the following independent processes are defined: (a) point measure $\omega$ with law $\mathbf P$, (b) Poisson point process $\mathcal Y$ equipped with independent $[0,1]$-uniform marks and (c) Poisson point process $\mathcal Z$.
Let $F_\mathcal V$ be the forest defined from the point processes (a) and (b) as in Section~\ref{sec:forest} and denote by $\widehat{\mathbb P}_0$ the joint law of random element $\xi\in\Xi$ and a doubly infinite lazy random walk $w$ on $F_\mathcal V$ as defined in \eqref{def:mathbb-widehat-Px}. 
For $n\in\Z$, consider the event $\mathcal G_n$ that 
\begin{itemize}
\item
$C_\mathcal V(w(n))$ is unbounded and has type $\mathcal A$;
\item
there exists $z\in\mathcal Z\cap \mathsf B(w(n),R)$ such that 
\begin{itemize}
\item
$C_{\mathcal V(\omega|_{\mathsf B(z,2)^c})}(w(n))$ has type $\neg \mathcal A$ and 
$\mathbf P\big[\omega(\mathsf B(z,2))=0\,\big|\,\omega|_{\mathsf B(z,2)^c}\big]\geq \sigma$,
\item
$\mathsf B(z,3)$ intersects at most one connected component of $C_\mathcal V(w(n))\setminus C_\mathcal V(w(n),r)$ and 
$C_\mathcal V(w(n))\setminus \big(C_\mathcal V(w(n),r)\cup\mathsf B(z,3)\big)$ contains at least $3$ unbounded connected components that touch $C_\mathcal V(w(n),r)$.
\end{itemize}
\item
the past of the random walk $\{w(n'), n'<n\}$ is contained in an unbounded connected component of $C_\mathcal V(w(n))\setminus C_\mathcal V(w(n),r)$ that does not intersect $\mathsf B(z,3)$. 
\end{itemize}
As in the proof of (a), we obtain from Proposition~\ref{prop:rw-stationarity} that
$\inf\limits_{n\in\Z}\mathbb {\widehat P}_0[\mathcal G_n] = \mathbb {\widehat P}_0[\mathcal G_0]>0$.

\smallskip

Now, modulo the above modifications to the proof of (a), we prove that $\lim_{n\to\infty}\mathbb {\widehat P}_0[\mathcal G_n]=0$ in essentially the same way as in the proof of (a). 
Let us just remark that for the natural analogues of the events $\mathcal A_0$, $\mathcal A_0'$ and $\mathcal G_{n,z}'$ from the proof of (a), we estimate the probability of event $\mathcal A_0'\cap\mathcal G_{n,z}'$ from above by $1/\sigma$ multiplied by the probability of the intersection of events (1) $\mathcal A_0'$, (2) $C_{\mathcal V(\omega|_{\mathsf B(z,2)^c})}(0)$ has type $\neg \mathcal A$, (3) $\omega(\mathsf B(z,2))=0$ and (4) $w(n)\in \mathsf B(z,R)$. 
Since on this event $\omega|_{\mathsf B(z,2)^c}=\omega$, it implies that $\mathcal A_0^c$ occurs, and one concludes as in (a). We omit further details. 
\end{proof}

\begin{remark}\label{rem:indistinguishability-graphs}
The principal difference of our approach from the one of Lyons and Schramm in \cite{LS-Indistinguishability} is that we use a reversible random walk on the forest of trifurcations, which is automatically transient. The same idea can be naturally applied in the setting of \cite{LS-Indistinguishability}. Indeed, consider an automorphism-invariant insertion-tolerant percolation on a transitive unimodular graph and let $F$ be the forest of trifurcations (cf.\ \cite[p.\ 1352]{BLPS-critical}). Conditioned on vertex $o$ being a trifurcation, let $w$ be a two-sided reversible simple random walk on $F$ with $w(0)=o$ and let $e_n$ be the edge selected uniformly and independently from all the edges at distance at most $R$ from $w(n)$. 
In the notation of the proof of \cite[Theorem~3.3]{LS-Indistinguishability}, let $\mathscr G_n$ be the event that (a) $C(w(n))$ is infinite and has type $\mathscr A$, (b) $e_n$ is pivotal for $C(w(n))$ and $Z(e_n)\geq \sigma$ and (c) $e_n$ intersects at most one of connected components of $C(w(n))\setminus\{w(n)\}$ and the past of the random walk $\{w(n'), n'<n\}$ is contained in an infinite component of $C(w(n))\setminus\{w(n)\}$ that does not intersect $e_n$. The probability of event $\mathscr G_n$ does not depend on $n$ and is positive. One can use these events instead of events $\mathscr B_n$ (which only differ from $\mathscr G_n$ in part (c), see just below \cite[(3.7)]{LS-Indistinguishability}) in the proof of \cite[Theorem~3.3]{LS-Indistinguishability}. 
The advantage of using the random walk on $F$ rather than a nearest-neighbor random walk on $C(o)$ is that the former is obviously transient. Thus, one can prove indistinguishability of infinite clusters without needing to know their transience. 
\end{remark}

\smallskip

For our applications, it is useful to have the following generalization of Theorem~\ref{thm:Indistinguishability} to marked point processes resp.\ to tuples of point processes; cf.\ \cite[Remark~3.4]{LS-Indistinguishability}. Its proof follows the proof of Theorem~\ref{thm:Indistinguishability} with only minor adjustments. 

\begin{remark}\label{rem:Indistinguishability-extension}
Let $\mathbf P$ be a probability measure on $\mathsf M(\mathsf X\times[0,1])\times\mathsf M(\mathsf X)$. It is insertion-tolerant (resp.\ deletion-tolerant) if for every $B\in\mathscr B_0(\mathsf X)$ the law of $(\widetilde\omega+\delta_{(X,U)},\omega')$ (resp.\ the law of $(\widetilde\omega|_{B^c\times[0,1]},\omega')$) is absolutely continuous with the law $\mathbf P$ of $(\widetilde\omega,\omega')$, where $(X,U)$ is independent from $(\widetilde\omega,\omega')$, $X$ is uniformly distributed on $B$ and $U$ is uniformly distributed on $[0,1]$. A set $\mathcal A\in\mathscr B(\mathsf X)\otimes\mathscr M(\mathsf X\times[0,1])\otimes\mathscr M(\mathsf X)$ is an occupied (resp.\ vacant) component property if $(x,\widetilde\omega,\omega')\in\mathcal A$ implies that $(x',\widetilde\omega,\omega')\in\mathcal A$ for all $x'$ connected to $x$ in $\mathcal O(\omega)$ (resp.\ $\mathcal V(\omega)$), where $\omega = \sum_{i\geq 1}\delta_{x_i}$ if $\widetilde\omega= \sum_{i\geq 1}\delta_{(x_i,u_i)}$. The statement of Theorem~\ref{thm:Indistinguishability} remains true for these generalizations. 
\end{remark}

\begin{proof}[Proof of Proposition~\ref{prop:Indistinguishability-hyperbolic}]
By \cite[Theorem~4.5.4]{Ratcliffe}, there exists an isometry from $\mathbb H^d$ to the unit ball of Euclidean space $\R^d$ so that the image of any ball in $\mathbb H^d$ is a Euclidean ball. 
On the other hand, by \cite[Proposition~5.4]{MeesterRoy}, if $\mathcal B$ is a collection of balls in Euclidean space $\R^d$, then  
the number of connected components in $S\setminus\big(\bigcup_{B\in\mathcal B}B\big)$ is dominated by $C_{S,d}\, k^d$, where $k$ is the number of balls from $\mathcal B$ intersecting $S\subset\R^d$. Because of the isometry, the same bound holds for the number of connected components of $\mathcal V\cap S$ for $S\subset\mathsf X$. Finally, the number of balls intersecting $\mathsf B(0,1)$ is precisely $\omega(\mathsf B(0,2))$. The result follows from the moment assumption. 
\end{proof}

\section{Proof of Theorem~\ref{thm:monotonicity-of-uniqueness}}\label{sec:uniqueness-monotonicity}

We begin with the proof of part (a). Let $\mathbf P$ be the law of a Poisson point measure $\widetilde\omega=\sum_{i\geq 1}\delta_{(x_i,u_i)}$ on $\mathsf X\times[0,1]$ with intensity $\lambda_2\mu_\mathsf X\otimes\mathrm{Leb}_{[0,1]}$. Then $\omega = \sum_{i\geq 1}\delta_{x_i}$ is a Poisson point measure on $\mathsf X$ with intensity $\lambda_2\mu_\mathsf X$ and $\omega' =\sum_{i\geq1\,:\,u_i\leq \lambda_1/\lambda_2}\delta_{x_i}$ is a Poisson point measure on $\mathsf X$ with intensity $\lambda_1\mu_\mathsf X$.
Define $\mathcal O_1 = \mathcal O(\omega')$ and $\mathcal O_2 = \mathcal O(\omega)$. 
Note that $\mathcal O_1\subseteq\mathcal O_2$. 

Let $\mathcal A\in\mathscr B(\mathsf X)\otimes\mathscr M(\mathsf X\times[0,1])$ be such that $(x,\widetilde\omega)\in\mathcal A$ if $C_{\mathcal O_2}(x)$ contains an unbounded connected component of $\mathcal O_1$. Note that $\mathcal A$ is an occupied component property in the sense of Remark~\ref{rem:Indistinguishability-extension}, which is invariant under the action of isometries of $\mathsf X$ (here $\gamma\widetilde\omega = \sum_{i\geq 1}\delta_{(\gamma x_i,u_i)}$). 
By Theorem~\ref{thm:Indistinguishability} and Remark~\ref{rem:Indistinguishability-extension}, either all unbounded components of $\mathcal O_2$ have type $\mathcal A$ or none of them has type $\mathcal A$ $\mathbf P$-almost surely. 
Since by assumption $N_{\mathcal O_1}=1$ almost surely and $\mathcal O_1\subseteq\mathcal O_2$, at least one of the infinite components of $\mathcal O_2$ has type $\mathcal A$ almost surely. Thus, all of them do. Hence $N_{\mathcal O_2}=1$ and (a) is proven. 

\smallskip

The proof of (b) is similar. Let $\mathbf P$ be the law of a Poisson point measure $\omega$ on $\mathsf X$ with intensity $\lambda_1\mu_\mathsf X$ and let $\omega'$ be an independent Poisson point measure on $\mathsf X$ with intensity $(\lambda_2-\lambda_1)\mu_\mathsf X$. Note that $\omega+\omega'$ is a Poisson point measure on $\mathsf X$ with intensity $\lambda_2\mu_\mathsf X$. Let $\mathcal V_1 = \mathcal V(\omega)$ and $\mathcal V_2 = \mathcal V(\omega+\omega')$. Note that $\mathcal V_1\supseteq\mathcal V_2$. 

Let $\mathcal A\in\mathscr B(\mathsf X)\otimes\mathscr M(\mathsf X)^{\otimes 2}$ be such that $(x,\omega,\omega')\in\mathcal A$ if $C_{\mathcal V_1}(x)$ contains an unbounded connected component of $\mathcal V_2$. Note that $\mathcal A$ is a vacant component property in the sense of Remark~\ref{rem:Indistinguishability-extension}, which is invariant under the action of isometries of $\mathsf X$. By Theorem~\ref{thm:Indistinguishability} and Remark~\ref{rem:Indistinguishability-extension}, either all unbounded components of $\mathcal V_1$ have type $\mathcal A$ or none of them has type $\mathcal A$ $\mathbf P$-almost surely. Argueing as in the proof of (a), we conclude that $N_{\mathcal V_1}=1$.
\qed

\section{Proof of Theorem~\ref{thm:connectivity}}\label{sec:connectivity}

Let $\mathcal S'$ be a random closed or open subset of $\mathsf X$ with an isometry invariant law and let $\{X_n\}_{n\geq 0}$ be an independent from $\mathcal S'$ random walk on $\mathsf X$ with $X_0=0$ and $X_n$ independently uniformly distributed in $\mathsf B(X_{n-1},1)$. By the ergodic theorem and L\'evy's $0$-$1$ law (precisely as in the proof of \cite[Lemma~8.2]{HaggstromJonasson}), the limit 
\begin{equation}\label{eq:frequency}
\lim\limits_{n\to\infty}\frac1n\sum\limits_{i=1}^n\mathds{1}_{\mathcal S'}(X_i)
\end{equation}
almost surely exists and only depends on $\mathcal S'$. 

For $x\in\mathsf X$, the component frequency of $C_\mathcal S(x)$ is defined (as in \cite[Definition~8.3]{HaggstromJonasson}) by 
\[
\lim\limits_{n\to\infty}\frac1n\sum\limits_{i=1}^n\mathds{1}_{C_\mathcal S(x)}(X_i)
\]
if the limit almost surely exists and does not depend on the random walk. We claim that if $\mu_\mathsf X(\partial\mathcal S)=0$ almost surely (which is the case if $\mathcal S$ is the occupied or vacant set of a Boolean model), then every connected component in $\mathcal S$ has a well defined component frequency. Indeed, the proof is essentially the same as that of \cite[Theorem~8.4]{HaggstromJonasson}. Let $\mathcal Z$ be an independent Poisson point process on $\mathsf X\times\R_+$ with intensity $\mu_\mathsf X\otimes\mathrm{Leb}$. For a connected component $\mathcal C$ in $\mathcal S$, we denote by $\mathcal Z_\mathcal C$ the set of all $x\in\mathcal C$ such that $(x,u)\in\mathcal Z$ for some $u>0$. Note that the closures of $\mathcal Z_\mathcal C$ and $\mathcal C$ coincide. 

To each $\mathcal Z_\mathcal C$, we attach an independent mark $\zeta_\mathcal C'$ equal to $1$ with probabilty $\frac12$ or $0$ otherwise, and denote by $\mathcal S'$ the closure of all $\mathcal Z_\mathcal C$'s with marks equal to $1$. Then the limit in \eqref{eq:frequency} exists and only depends on $\mathcal S'$. Now, fix $x\in\mathsf X$ and consider the marks $\zeta_\mathcal C''$ with 
$\zeta_{C_\mathcal S(x)}'' = 1 - \zeta_{C_\mathcal S(x)}'$ and $\zeta_\mathcal C'' = \zeta_\mathcal C'$ if $\mathcal C\neq C_\mathcal S(x)$ and the corresponding subset $\mathcal S''$. Note that $\mathcal S'$ and $\mathcal S''$ have the same distribution, hence the limit \eqref{eq:frequency} exists also with $\mathcal S'$ replaced by $\mathcal S''$ and depends only on $\mathcal S''$. Thus, the limit
\[
\Big|\lim\limits_{n\to\infty}\frac1n\sum\limits_{i=1}^n\big(\mathds{1}_{\mathcal S'}(X_i)-\mathds{1}_{\mathcal S''}(X_i)\big)\Big|
=
\lim\limits_{n\to\infty}\frac1n\sum\limits_{i=1}^n\big|\mathds{1}_{\mathcal S'}(X_i)-\mathds{1}_{\mathcal S''}(X_i)\big|
=\lim\limits_{n\to\infty}\frac1n\sum\limits_{i=1}^n\mathds{1}_{C_\mathcal S(x)}(X_i)
\]
exists and is independent of the random walk, 
where the first equality holds, since almost surely either $\mathcal S'\subseteq \mathcal S''$ or $\mathcal S''\subseteq \mathcal S'$, and the second, since $\mu_\mathsf X(\partial C_\mathcal S(x))=0$ almost surely.

\smallskip

We proceed with the proof of part (a) and omit the proof of part (b), since it is essentially the same. 
By Theorem~\ref{thm:Indistinguishability}, all unbounded occupied components have the same component frequencies almost surely, and it equals to a constant $c\in[0,1]$ almost surely by the ergodicity assumption. If $c>0$, then a similar ``glueing argument'' as in the proof of part (a) of Lemma~\ref{l:number-of-infinite-clusters-occupied} (using the assumption $N_\mathcal O=\infty$) allows to conclude that with positive probability there exists an unbounded component with component frequency $\geq 2c\neq c$, which is impossible. Thus, $c=0$. Also, the component frequency of all bounded occupied components must be $0$, since otherwise a similar glueing argument would prove the existence of an unbounded component with positive frequency. All in all, the component frequency of $C_\mathcal O(0)$ is almost surely $0$. Thus, by the bounded convergence theorem, 
\[
0 = \mathsf E\Big[\lim\limits_{n\to\infty}\frac1n\sum\limits_{i=1}^n\mathds{1}_{C_\mathcal O(0)}(X_i)\Big]=\lim\limits_{n\to\infty}\frac1n\sum\limits_{i=1}^n\mathsf P\big[X_i\in C_\mathcal O(0)\big]
\geq \inf\limits_{x,x'\in\mathsf X}\tau_\mathcal O(x,x').
\]
The proof is completed. \qed

\section{Poisson-Boolean model on random subsets of $\mathsf X$}\label{sec:percolation}

Let $\mathcal S$ be a random closed or open subset of $\mathsf X$ and let $\eta$ be an independent Poisson point measure on $\mathsf X\times\R_+$ with intensity $\mu_\mathsf X\otimes\mathrm{Leb}$. If $\eta = \sum_{i\geq 1}\delta_{(x_i,u_i)}$, we consider the Poisson point measure $\eta_\lambda = \sum_{i\geq 1,\,u_i\leq\lambda}\delta_{x_i}$ on $\mathsf X$ with intensity $\lambda\mu_\mathsf X$ and denote by $\mathcal Z_\lambda$ its support. Note that $\mathcal Z_{\lambda_1}\subseteq\mathcal Z_{\lambda_2}$ for all $\lambda_1<\lambda_2$. 

Consider the random graph $G_{\mathcal S,\lambda}$ with vertex set $V_{\mathcal S,\lambda} = \mathcal Z_\lambda\cap\mathcal S$ and edge set $E_{\mathcal S,\lambda}$ consisting of all pairs of vertices $z,z'$ whose mutual distance within $\mathcal S$ is at most $2$. In this section, we are interested in the existence of infinite connected components in $G_{\mathcal S,\lambda}$. 

\begin{theorem}\label{thm:percolation}
Assume that the law of $\mathcal S$ is invariant under isometries of $\mathsf X$ and the expected number of connected components in $\mathcal S\cap\mathsf B(0,1)$ is finite. Let $r>0$. 
Let $\mathcal Y$ be the support of a simple point process on $\mathsf X$ with an isometry invariant law and independent from $\mathcal S$ and $\eta$.

Then almost surely, for every $r$-trifurcation $y\in\mathcal Y$ for $\mathcal S$, the restriction of $G_{\mathcal S,\lambda}$ to $C_\mathcal S(y)$ contains an unbounded connected component for all $\lambda$ large enough. 
\end{theorem}
Theorem~\ref{thm:percolation-intro} directly follows from Theorem~\ref{thm:percolation}, Lemma~\ref{l:infinite-clusters-trifurcations-2} and Theorem~\ref{thm:Indistinguishability}. Indeed, if $\mathcal S$ is the occupied set of an isometry invariant insertion-tolerant Boolean model or the vacant set of an isometry invariant deletion-tolerant Boolean model with almost surely infinitely many unbounded components, then by Lemma~\ref{l:infinite-clusters-trifurcations-2} and Theorem~\ref{thm:percolation}, for every unbounded component of $\mathcal S$, the restriction of $G_{\mathcal S,\lambda}$ to this component contains an infinite connected component for all large enough $\lambda$. 
By Theorem~\ref{thm:Indistinguishability} and Remark~\ref{rem:Indistinguishability-extension}, almost surely for each $\lambda$, either the restriction of $G_{\mathcal S,\lambda}$ to every unbounded component of $\mathcal S$ contains an infinite component or to none of them. Finally, the existence of (non-random) $\lambda_*$ follows from the ergodicity assumption.

\begin{proof}[Proof of Theorem~\ref{thm:percolation}]
The proof is inspired by \cite[Section~4]{BLPS-critical}. 
Denote by $\mathcal Y_1$ the set of all $r$-trifurcations $y\in\mathcal Y$, such that the restriction of $G_{\mathcal S,\lambda}$ to $C_\mathcal S(y)$ consists only of finite components for all $\lambda$, and assume that with positive probability $\mathcal Y_1\neq\emptyset$; in particular, with positive probability $\mathcal Y_1\cap\mathsf B(0,1)\neq\emptyset$. 

\smallskip

Recall the definition of random forest $F_\mathcal S$ from Section~\ref{sec:forest}. By Lemma~\ref{l:infinite-clusters-trifurcations-2}, on event $\mathcal Y_1\neq\emptyset$, $F_\mathcal S$ is non-empty and every vertex of $F_\mathcal S$ has degree at least $3$. For each $\lambda$, consider the random subforest $F_\mathcal S^\lambda$ of $F_\mathcal S$ obtained by retaining every edge $\{y,y'\}$ of $F_\mathcal S$ if there exist $z$ and $z'$ from a same connected component of $G_{\mathcal S,\lambda}$, such that $z$ is within distance $1$ from $y$ in $\mathcal S$ and $z'$ is within distance $1$ from $y'$ in $\mathcal S$. Note that every $y\in\mathcal Y_1$ is in a finite tree of $F_\mathcal S^\lambda$ for all $\lambda$. 

\smallskip

Denote by $\mathrm{deg}_\mathcal S^\lambda(y)$ the degree of $y$ in $F_\mathcal S^\lambda$ and by $K^\lambda_\mathcal S(y)$ the connected component of $y$ in $F_\mathcal S^\lambda$ and consider the random function $m^\lambda:\mathcal Y_1\times\mathcal Y_1\to\R_+$ defined by 
\[
m^\lambda(y,y') = \frac{\mathrm{deg}_\mathcal S^\lambda(y)}{|K^\lambda_\mathcal S(y)|} 
\]
if $y'\in K^\lambda_\mathcal S(y)$ and $m^\lambda(y,y')=0$ otherwise.
Note that for all $y\in\mathcal Y_1$, 
\[
\sum\limits_{y'\in \mathcal Y_1}m^\lambda(y,y')=\mathrm{deg}_\mathcal S^\lambda(y)
\]
and 
\[
\sum\limits_{y'\in \mathcal Y_1}m^\lambda(y',y)=\frac{1}{|K^\lambda_\mathcal S(y)|}\,\sum\limits_{y'\in K^\lambda_\mathcal S(y)}\mathrm{deg}_\mathcal S^\lambda(y')
= \frac{2(|K^\lambda_\mathcal S(y)|-1)}{|K^\lambda_\mathcal S(y)|}
\leq 2.
\]
By the mass-transport principle (Lemma~\ref{l:mtp}), 
\begin{eqnarray*}
\mathsf E\Big[\sum\limits_{y\in \mathcal Y_1\cap\mathsf B(0,1)}\mathrm{deg}_\mathcal S^\lambda(y)\Big]
&= 
&\mathsf E\Big[\sum\limits_{y\in \mathcal Y_1\cap\mathsf B(0,1)}\sum\limits_{y'\in \mathcal Y_1}m^\lambda(y,y')\Big]
=\mathsf E\Big[\sum\limits_{y\in \mathcal Y_1\cap\mathsf B(0,1)}\sum\limits_{y'\in \mathcal Y_1}m^\lambda(y',y)\Big]\\
&\leq &2\,\mathsf E\big[|\mathcal Y_1\cap\mathsf B(0,1)|\big].
\end{eqnarray*}
On the other hand, since $\lim_{\lambda\to\infty}\mathrm{deg}^\lambda_\mathcal S(y)\geq 3$, by the monotone convergence theorem, 
\[
\lim\limits_{\lambda\to\infty}\mathsf E\Big[\sum\limits_{y\in \mathcal Y_1\cap\mathsf B(0,1)}\mathrm{deg}_\mathcal S^\lambda(y)\Big]
\geq 3\,\mathsf E\big[|\mathcal Y_1\cap\mathsf B(0,1)|\big].
\]
This is a contradiction. Thus, $\mathcal Y_1=\emptyset$ almost surely and the proof is completed. 
\end{proof}

\begin{remark}
It can be shown that all infinite components of $G_{\mathcal S,\lambda}$ are transient almost surely (using similar ideas as in Section~\ref{sec:transience}). Then one could use a lazy nearest-neighbor random walk on $G_{\mathcal S,\lambda}$ instead of the random walk on the forest $F_\mathcal S$ for a proof of  Theorem~\ref{thm:Indistinguishability}. In this case, instead of the events $\mathcal G_n$ one could consider events that are more similar to the original events $\mathscr B_m$ in \cite{LS-Indistinguishability} (see just below (3.7) there). 
\end{remark}

\section{Transience of unbounded occupied components}\label{sec:transience}

Let $G=(V_G,E_G)$ be a random unoriented graph with vertices in $\mathsf X$, as defined in Section~\ref{sec:rw}.
In this section, we study transience of infinite connected components of $G$. 

\smallskip

For $x\in V_G$ and $r>0$, we denote by $C_G(x)$ the connected component (cluster) of $x$ in $G$ and by $C_G(x,r)$ the connected component of $x$ in the maximal subgraph of $G$ with the vertex set $V_G\cap \mathsf B(x,r)$. We say that $y\in V_G$ is $r$-\emph{trifurcation} if $C_G(y)\setminus C_G(y,r)$ contains at least $3$ infinite connected components. 

\begin{theorem}\label{thm:rw-transience}
Let $G$ be a random unoriented graph with an isometry invariant law. 
Almost surely, for every $r$-trifurcation $y\in V_G$, $C_G(y)$ is transient. 
\end{theorem}
Theorem~\ref{thm:rw-transience-intro} easily follows from Theorem~\ref{thm:rw-transience}. Indeed, one shows similarly to the proof of Lemma~\ref{l:infinite-clusters-trifurcations-2} that under the assumptions of Theorem~\ref{thm:rw-transience-intro}, every infinite connected component $C$ of $G$ contains $r$-trifurcations for some $r=r(C)$. 

\smallskip

Before proving Theorem~\ref{thm:rw-transience}, we discuss some properties of trifurcations. 
\begin{lemma}\label{l:graph-trifurcation}
Let $G$ be a random unoriented graph with an isometry invariant law. Let $r>0$ and $M<\infty$.  Almost surely if $y$ is $r$-trifurcation with $|V_G\cap\mathsf B(y,2r)|<M$, then every infinite component of $C_G(y)\setminus C_G(y,r)$ contains infinitely many $r$-trifurcations $y'$ with $|V_G\cap\mathsf B(y',2r)|<M$.
\end{lemma}
\begin{proof}
The same argument as in the proof of Lemma~\ref{l:tree-of-trifurcations} applies. We omit the details. 
\end{proof}

\smallskip

Consider a random unoriented graph $G$ with independent $[0,1]$-uniform random labels $\{u_e\}_{e\in E_G}$ (see e.g.\ \cite[Section~5.2]{LastPenrose} for a construction of marked point processes). The minimal spanning forest $F_G$ of $G$ is a subgraph of $G$ obtained by 
deleting from every cycle $(x_{i_1},\dots, x_{i_k})$ of $G$ the edge with the maximal label. 
By construction, (a) every connected component of $F_G$ is a tree and (b) every tree of $F_G$ contained in an infinite connected component of $G$ is infinite. Moreover, if $G$ has an isometry invariant law then so does $F_G$.

\smallskip

A vertex $x$ of a tree $T$ is a trifurcation if there are at least $3$ edge disjoint infinite paths from $x$ in $T$. In general, a $r$-trifurcation of $G$ is not a trifurcation of $F_G$, but the following holds.
\begin{lemma}\label{l:tree-trifurcation}
Let $G$ be a random labeled unoriented graph with an isometry invariant law and $r>0$.  Almost surely, if $y$ is $r$-trifurcation for $G$ then $C_G(y)$ contains a tree of $F_G$ with at least one trifurcation. 
\end{lemma}
\begin{proof}
The proof is essentially the same as the proof of \cite[Lemma~8.35]{LyonsPeresBook}. 
Consider a realization of labeled graph $G$ and let $y$ be a $r$-trifurcation of $G$. 
Let $T$ be an arbitrary (finite) spanning tree of $C_G(y,r)$. Note that if $u_e<\frac12$ for all edges $e$ of $T$ and $u_e>\frac12$ for all edges $e$ on the boundary of $T$ in $G$, which occurs with (conditional) probability 
$\geq \big(\frac12\big)^{|V_G\cap\mathsf B(y,2r)|^2}$, then 
there is a tree in the minimal spanning forest $F_G$ which contains $T$ and has infinite intersection with every infinite component of $C_G(y)\setminus C_G(y,r)$. Thus, at least one of the vertices of $T$ is a trifurcation of $F_G$. 
By Lemma~\ref{l:graph-trifurcation}, every connected component of $G$ with $r$-trifurcation contains infinitely many $r$-trifurcations $y'$ with $|V_G\cap\mathsf B(y',2r)|<M$ for some $M$, thus it contains a trifurcation of $F_G$ almost surely. 
\end{proof}

\begin{proof}[Proof of Theorem~\ref{thm:rw-transience}]
By the monotonicity of transience (see e.g.\ \cite[Section~2.4]{LyonsPeresBook}) and Lemma~\ref{l:tree-trifurcation}, it suffices to prove that every tree of $F_G$ which contains a trifurcation is transient. Denote the forest of these trees by $F_G'$.
We further prune the trees in $F_G'$ by cutting off their dangling ends. More precisely, we define \emph{backbone} $B(T)$ of tree $T$ as the subtree of $T$ induced by the vertices of $T$ from which there exist at least two edge disjoint infinite paths in $T$ and let $\widehat F_G = \{B(T):T\in F_G'\}$ be the forest consisting of backbones of the trees from $F_G'$. The result will follow, if we show that 
 \begin{equation}\label{eq:transience-occupied-forest}
\text{every tree in $\widehat F_G$ is transient almost surely.}
\end{equation}
It is possible to prove \eqref{eq:transience-occupied-forest} by adapting the ideas from \cite[Sections~8.3 and 8.6]{LyonsPeresBook}, but we propose here a more direct proof by constructing a flow of finite energy from a trifurcation on each of the trees in $\widehat F_G$.

\smallskip

We write $\widehat V_G$ for the vertex set of $\widehat F_G$ and $\widehat T_G$ for the set of trifurcations. For $x\in \widehat V_G$, we denote by $D_x$ the degree of $x$ in $\widehat F_G$. Note that $D_x\geq 3$ if $x\in\widehat T_G$ and $D_x=2$ otherwise.
For each realization of the forest $\widehat F_G$ and any $y\in\widehat T_G$, we define the function $\theta_y :\widehat V_G\to \mathbb R_+$ by $\theta_y(y) = 1$, $\theta_y(x)=0$ if $x$ is not in the same tree as $y$ and 
\[
\theta_y(x) = \frac{1}{D_y}\prod\limits_{z\in\pi_{yx}\setminus\{x,y\}}\frac{1}{D_z-1},
\]
if $x$ is in the tree of $y$ in $\widehat F_G$, where $\pi_{yx}$ is the unique path from $y$ to $x$ in $\widehat F_G$. Note that $\theta_y$ is a unit flow from $y$ (which splits uniformly at trifurcations) and its energy equals
\[
\mathcal E(\theta_y) = \sum\limits_{x\in\widehat V_G}\theta_y(x)^2.
\]
We claim that $\mathcal E(\theta_y)<\infty$ almost surely. We write $\mathcal E(\theta_y) = \mathcal E_1(\theta_y) + \mathcal E_2(\theta_y)$, where 
\[
\mathcal E_1(\theta_y) = \sum\limits_{x\in\widehat T_G} \theta_y(x)^2\quad\text{and}\quad
\mathcal E_2(\theta_y) = \sum\limits_{x\notin\widehat T_G} \theta_y(x)^2.
\]
For $x\in \widehat V_G$ and $n\geq 1$, let $T_n(x)$ be the set of all trifurcations $y'$ of $\widehat F_G$ such that the unique path in $\widehat F_G$ from $x$ to $y'$ contains exactly $n$ trifurcations (including $y'$, but not $x$). 
Since the forest $\widehat F_G$ has an isometry invariant law, by Lemma~\ref{l:graph-trifurcation} applied to $\widehat F_G$, every infinite path in every tree of $\widehat F_G$ contains infinitely many trifurcations almost surely. In particular, 
$|T_1(x)| = D_x$ and $|T_n(x)|\geq 2|T_{n-1}(x)|$ for $n\geq 2$. 

\smallskip

We first show that $\mathcal E_1(\theta_y)<\infty$ almost surely. We have 
\begin{eqnarray*}
\mathcal E_1(\theta_y) &= &1 + \sum\limits_{n=1}^\infty \sum\limits_{y'\in T_n(y)}\theta_y(y')^2\\
&= &1 + \sum\limits_{n=1}^\infty \sum\limits_{y_1\in T_1(y)}\sum\limits_{y_2\in T_1(y_1)\setminus\{y\}}\ldots \sum\limits_{y_n\in T_1(y_{n-1})\setminus\{y_{n-2}\}}
\Big(\frac{1}{D_y}\prod\limits_{i=1}^{n-1}\frac{1}{D_{y_i}-1}\Big)^2\\
&= &1 + \sum\limits_{n=1}^\infty \frac{1}{D_y^2}\sum\limits_{y_1\in T_1(y)}\frac{1}{(D_{y_1}-1)^2}\sum\limits_{y_2\in T_1(y_1)\setminus\{y\}}\ldots \frac{1}{(D_{y_{n-1}}-1)^2}
\sum\limits_{y_n\in T_1(y_{n-1})\setminus\{y_{n-2}\}}1\\
&\leq & 1 + \sum\limits_{n=1}^\infty\frac{1}{2^n}<\infty.
\end{eqnarray*}
It remains to show that $\mathcal E_2(\theta_y)<\infty$ almost surely. Consider the mass function 
\[
m(y,x) = \left\{\begin{array}{lll} \theta_y(x)^2 && y\in\widehat T_G, x\notin\widehat T_G\\ 0 && \text{else.}\end{array}\right.
\]
We claim that almost surely
\begin{equation}\label{eq:claim:energy}
\sum\limits_{y\in\widehat T_G} m(y,x)\leq 1\quad\text{for all } x\notin\widehat T_G
\end{equation}
(and $=0$ for $x\in\widehat T_G$). Assume \eqref{eq:claim:energy}. The function 
\[
\phi(A,B) = \mathsf E\big[\sum\limits_{y\in \widehat V_G\cap A}\sum\limits_{x\in\widehat V_G\cap B} m(y,x)\big] 
\]
is isometry invariant and $\phi(\mathsf X,\mathsf B(0,1))\leq \mathsf E\big[|V_G\cap\mathsf B(0,1)|\big]<\infty$ by \eqref{eq:claim:energy}. Thus, by the mass-transport principle (Lemma~\ref{l:mtp}) and the fact that $\sum\limits_x m(y,x) = \mathcal E_2(\theta_y)$, 
\[
\phi(\mathsf B(0,1),\mathsf X) = \mathsf E\big[\sum\limits_{y\in\widehat T_G\cap\mathsf B(0,1)}\mathcal E_2(\theta_y)\big]<\infty.
\]
In particular, $\mathcal E_2(\theta_y)<\infty$ for all $y\in\widehat T_G$ almost surely. 

\smallskip

It remains to prove \eqref{eq:claim:energy}. Let $x\notin \widehat T_G$. We have 
\begin{eqnarray*}
\sum\limits_y m(y,x) &= &\sum\limits_{y\in \widehat T_G} \theta_y(x)^2
= \sum\limits_{n=1}^\infty \sum\limits_{y\in T_n(x)}\theta_y(x)^2\\
&= &\sum\limits_{n=1}^\infty \sum\limits_{y_1\in T_1(x)}\sum\limits_{y_2\in T_1(y_1)\setminus T_1(x)}\sum\limits_{y_3\in T_1(y_2)\setminus \{y_1\}}\hspace{-5pt}\ldots\hspace{-5pt} \sum\limits_{y_n\in T_1(y_{n-1})\setminus\{y_{n-2}\}}
\hspace{-5pt}\Big(\frac{1}{D_{y_n}}\prod\limits_{i=1}^{n-1}\frac{1}{D_{y_i}-1}\Big)^2\\
&= &\sum\limits_{n=1}^\infty \sum\limits_{y_1\in T_1(x)}\frac{1}{(D_{y_1}-1)^2}\sum\limits_{y_2\in T_1(y_1)\setminus T_1(x)}\frac{1}{(D_{y_2}-1)^2}\,\,\ldots 
\sum\limits_{y_n\in T_1(y_{n-1})\setminus\{y_{n-2}\}}\frac{1}{D_{y_n}^2}\\
&\leq & |T_1(x)|\sum\limits_{n=1}^\infty\frac{1}{2^{n+1}}= 1.
\end{eqnarray*}
The proof is completed. 
\end{proof}
\begin{remark}
If the stronger assumption $\mathsf E[|V_G\cap\mathsf B(0,2)|^2]<\infty$ holds, then one can estimate the energy $\mathcal E(\theta_y)$ directly by considering the function $m'(y,x) = \theta_y(x)^2$ for all $x$. Indeed, $\mathcal E(\theta_y) = \sum_x m'(y,x)$, the last estimate in the proof of Theorem~\ref{thm:rw-transience} gives $\sum_y m'(y,x) \leq \frac12 D_x$, and  
$\sum_{x\in V_G\cap\mathsf B(0,1)}D_x\leq |V_G\cap\mathsf B(0,2)|^2$, so the finiteness of $\mathcal E(\theta_y)$ follows by an application of the mass-transport principle as in the estimation of $\mathcal E_2(\theta_y)$. 
\end{remark}

\section{Boolean models with random radii}\label{sec:random-radii}

As mentioned in the introduction, all the results of this paper extend to Boolean models with random i.i.d.\ radii driven by isometry invariant insertion- or deletion-tolerant processes, but some notation and proof steps get more involved. In this section, we discuss necessary modifications to definitions and proofs. 

\smallskip

We consider random marked point measures on $\mathsf X$ with i.i.d.\ positive marks (see e.g.\ 
\cite[Section~5.2]{LastPenrose} for a precise construction) and write $\widetilde \omega=\sum_{i\geq 1}\delta_{(x_i,r_i)}$ for their realizations. Each $\widetilde\omega$ induces the partition of $\mathsf X$ into the \emph{occupied set} $\mathcal O(\widetilde\omega) = \bigcup_{i\geq 1}\mathsf B(x_i,r_i)$ and the \emph{vacant set} $\mathcal V(\widetilde\omega) = \mathsf X\setminus\mathcal O(\widetilde\omega)$. 
The definitions of the occupied and vacant component properties extend naturally to the setting of random radii Boolean models.

We denote by $\mathsf Q$ the common law of the marks. 
We say that a random marked point measure $\widetilde \omega$ is \emph{insertion-tolerant}, if for every $B\in\mathscr B_0(\mathsf X)$, the law of $\widetilde\omega+\delta_{(X,\varrho)}$ is absolutely continuous with respect to the law of $\widetilde\omega$, where $(X,\varrho)$ is independent from $\widetilde\omega$, $X$ is uniformly distributed in $B$ and $\varrho$ has law $\mathsf Q$. 
We say that a random marked point measure $\widetilde \omega$ is \emph{deletion-tolerant}, if for every $B\in\mathscr B_0(\mathsf X)$, the law of $\widetilde\omega|_{B^c\times\R_+}$ is absolutely continuous with respect to the law of $\widetilde\omega$. 

\smallskip

While no additional assumptions are needed to extend our results about the occupied set to the setting of random radii, our methods allow to extend our results about the vacant set only under the additional assumption that 
\begin{equation}\label{eq:assumption-iid-vacant}
\text{the number of balls intersecting $\mathsf B(0,1)$ is finite almost surely,}
\end{equation}
which we need in order to be able to make local modifications in the vacant set.
For the Poisson-Boolean model with random radii, condition \eqref{eq:assumption-iid-vacant} holds iff $\mathsf E[\mu_\mathsf X(\mathsf B(0,\varrho))]<\infty$, where $\varrho$ is a generic random variable with law $\mathsf Q$ (see e.g.\ \cite[Section~3.1]{MeesterRoy}), which is also equivalent to the vacant set being non-empty. Thus, condition \eqref{eq:assumption-iid-vacant} is not restrictive for Poisson-Boolean models, but we do not know if for general isometry invariant deletion-tolerant Boolean models, non-emptiness of the vacant set implies \eqref{eq:assumption-iid-vacant} or not. 

\begin{theorem}\label{thm:main-iid}
Let $\widetilde\omega$ be a random point measure on $\mathsf X$ with i.i.d.\ positive marks and an isometry invariant law. 
\begin{itemize}
 \item[(a)]
The analogues of Theorems~\ref{thm:Indistinguishability}, \ref{thm:monotonicity-of-uniqueness}--\ref{thm:rw-transience-intro} hold for the occupied set $\mathcal O(\widetilde\omega)$. 
\item[(b)]
The analogues of Theorems~\ref{thm:Indistinguishability}, \ref{thm:monotonicity-of-uniqueness}--\ref{thm:percolation-intro} hold for the vacant set $\mathcal V(\widetilde\omega)$, if, in addition to the assumptions of those theorems, $\widetilde\omega$ satisfies \eqref{eq:assumption-iid-vacant}. 
\end{itemize}
\end{theorem}
The proof of Theorem~\ref{thm:main-iid} follows the same line as the proof of its special case, when all radii are equal to $1$, considered in details in the previous sections, but some notation and proof steps get more involved. In fact, we only have to suitably adjust the statements and/or the proofs of Lemmas~\ref{l:conditional-probability-positive}, \ref{l:number-of-infinite-clusters-occupied}, \ref{l:infinite-clusters-trifurcations-2} and \ref{l:pivotals} and Theorem~\ref{thm:Indistinguishability} and the definition of pivotal sets. In the rest of this section, we discuss the precise modifications that should be made there. 

\smallskip

A substitute for Lemma~\ref{l:conditional-probability-positive} is the following lemma. 
\begin{lemma}\label{l:conditional-probability-positive-iid}
Let $\widetilde B=B\times[0,r_*]$, for $B\in\mathscr B_0(\mathsf X)$ and $r_*>0$ such that $\mathsf P[\varrho\leq r_*]>0$. 
\begin{itemize}\itemsep4pt
\item[(a)]
If $\widetilde \omega$ is insertion-tolerant, then for every $\sigma(\widetilde\omega|_{\widetilde B^c})$-measurable set $\widetilde S\subseteq \widetilde B$ such that $(\mu_\mathsf X\otimes\mathsf Q)(\widetilde S)>0$ a.s.,
$\mathsf P[\widetilde\omega(\widetilde S)=\widetilde\omega(\widetilde B)=1\,|\,\widetilde\omega|_{\widetilde B^c}] >0$ a.s.\ on the event $\widetilde \omega(\widetilde B)=0$.
\item[(b)]
If $\widetilde \omega$ is deletion-tolerant, then 
$\mathsf P[\widetilde \omega(\widetilde B)=0\,|\,\widetilde\omega|_{\widetilde B^c}] >0$ a.s.\ 
on the event $\widetilde\omega(B\times(r_*,\infty))=0$. 
\end{itemize}
\end{lemma}
\begin{proof}
The proof is essentially the same as the proof of Lemma~\ref{l:conditional-probability-positive}. Assume there exists $A$ such that 
$\mathsf P[\widetilde\omega|_{\widetilde B^c}\in A, \widetilde w(\widetilde B)=0]>0$ and $\mathsf P[\widetilde \omega(\widetilde S)=\widetilde w(\widetilde B)=1,\widetilde \omega|_{\widetilde B^c}\in A]=0$, where $\widetilde S$ is as in (a). By the insertion-tolerance of $\widetilde \omega$, if $(X,\varrho)$ is independent from $\widetilde\omega$, $X$ is uniformly distributed in $B$ and $\varrho$ has law $\mathsf Q$, then 
\begin{eqnarray*}
0 &= &\mathsf P\big[(\widetilde \omega+\delta_{(X,\varrho)})(\widetilde S)=(\widetilde \omega+\delta_{(X,\varrho)})(\widetilde B)=1,(\widetilde \omega+\delta_{(X,\varrho)})|_{\widetilde B^c}\in A\big]\\
&\geq &\mathsf P\big[(X,\varrho)\in\widetilde S, \widetilde w(\widetilde B)=0 ,\widetilde \omega|_{\widetilde B^c}\in A\big]
= \mathsf E\Big[\tfrac{(\mu_\mathsf X\otimes\mathsf Q)(\widetilde S)}{\mu_\mathsf X(B)}\,\mathds{1}_{\{\widetilde \omega(\widetilde B)=0,\,\widetilde\omega|_{\widetilde B^c}\in A\}}\Big] >0.
\end{eqnarray*}
This contradiction proves (a). 
Now let $B' = B\times(r_*,\infty)$ and assume there exists $A$ such that $\mathsf P[\widetilde\omega|_{\widetilde B^c}\in A, \widetilde\omega(B')=0]>0$ and $\mathsf P[\widetilde \omega(\widetilde B)=0,\widetilde\omega|_{\widetilde B^c}\in A]=0$. 
By the deletion-tolerance of $\widetilde\omega$, 
\begin{eqnarray*}
0 &= &\mathsf P\big[\widetilde\omega|_{B^c\times\R_+}(\widetilde B)=0,(\widetilde\omega|_{B^c\times\R_+})|_{\widetilde B^c}\in A\big]
= \mathsf P[\widetilde\omega|_{B^c\times\R_+}\in A]\\
&\geq &\mathsf P[\widetilde\omega|_{\widetilde B^c}\in A, \widetilde\omega(B')=0]>0.
\end{eqnarray*}
This contradiction proves (b).
\end{proof}

\smallskip

The method of ``glueing'' components together, used in the proofs of Lemmas~\ref{l:number-of-infinite-clusters-occupied} and \ref{l:infinite-clusters-trifurcations-2}, extends naturally to the setting of random radii. We only explain adjustments to the proof of Lemma~\ref{l:number-of-infinite-clusters-occupied}, but exactly the same modifications apply to the proof of Lemma~\ref{l:infinite-clusters-trifurcations-2}. First, assume that $\mathsf B(0,n)$ is intersected by $k$ unbounded occupied components with positive probability. We choose $r_1>0$ such that $\mathsf P[\varrho>r_1]>0$ and consider a finite covering of $\mathsf B(0,n)$ by balls $B_i$ of radius $\frac12r_1$. Then with positive probability, the enhanced occupied set $\mathcal O\cup\bigcup_i \mathsf B(X_i,\varrho_i)$ contains a unique unbounded component, where $(X_i,\varrho_i)$ are independent, $X_i$ is uniformly distributed in $B_i$ and $\varrho_i$ has law $\mathsf Q$, which leads to a contradiction as in the proof of Lemma~\ref{l:number-of-infinite-clusters-occupied}. 
Now, assume that with positive probability, $\mathsf B(0,n)$ is intersected by $k$ unbounded vacant components and all the balls intersecting $\mathsf B(0,n)$ have radius smaller than $r_2$ (such $r_2$ exists by assumption \eqref{eq:assumption-iid-vacant}). Then with positive probability, the enhanced vacant set $\mathcal V(\widetilde\omega|_{\mathsf B(0,n+r_2)^c\times\R_+})$ contains a unique unbounded component, which leads to a contradiction as in the proof of Lemma~\ref{l:number-of-infinite-clusters-occupied}. 

\smallskip

Next, we suitably generalize the definition of the pivotal set from Section~\ref{sec:pivotal}. Let $r_*>0$. Given an occupied component property $\mathcal A$, a set $B\in\mathscr B_0(\mathsf X)$ is called \emph{$r_*$-pivotal} for the occupied connected component of $x$ if $\widetilde \omega(B\times[0,r_*])=0$ and there exists a measurable set $\widetilde S=\widetilde S(\widetilde \omega)\subseteq B\times[0,r_*]$ with $(\mu_\mathsf X\otimes\mathsf Q)(\widetilde S)>0$ a.s.\ such that precisely one of $(x,\widetilde\omega)$ and $(x,\widetilde\omega+\delta_{(y,r)})$ is in $\mathcal A$, for every $(y,r)\in \widetilde S$.
Given a vacant component property $\mathcal A$, a set $B\in\mathscr B_0(\mathsf X)$ is called \emph{$r_*$-pivotal} for the vacant connected component of $x$ if 
$\widetilde\omega(B\times(r_*,\infty))=0$ and 
precisely one of $(x,\widetilde\omega)$ and $(x,\widetilde\omega|_{(B\times[0,r_*])^c})$ is in $\mathcal A$.
The following lemma is a suitable substitute for Lemma~\ref{l:pivotals}.

\begin{lemma}\label{l:pivotals:iid}
Let $\widetilde\omega$ be a marked point measure with an isometry invariant law and let $\mathcal Z$ be an independent Poisson point process on $\mathsf X$ with intensity $\mu_\mathsf X$. 
\begin{itemize}\itemsep4pt
\item[(a)]
If $\mathcal A$ is an occupied component property, $\widetilde\omega$ is insertion-tolerant, and with positive probability there exist unbounded occupied components of both types $\mathcal A$ and $\neg\mathcal A$, then there exist $\delta$ and $r_*$ such that with positive probability, $C_\mathcal O(0)$ is unbounded and $\mathsf B(z,\delta)$ is $r_*$-pivotal for $C_\mathcal O(0)$ for some $z\in\mathcal Z$. 
\item[(b)]
If $\mathcal A$ is a vacant component property, $\widetilde\omega$ is deletion-tolerant and satisfies \eqref{eq:assumption-iid-vacant}, and with positive probability there exist unbounded vacant components of both types $\mathcal A$ and $\neg\mathcal A$, then
there exist $\Delta$ and $r_*$ such that with positive probability, $C_\mathcal V(0)$ is unbounded and $\mathsf B(z,\Delta)$ is $r_*$-pivotal for $C_\mathcal V(0)$ for some $z\in\mathcal Z$. 
\end{itemize}
\end{lemma}
\begin{proof} 
The proof of Lemma~\ref{l:pivotals:iid} is similar to the proof of Lemma~\ref{l:pivotals} and we only explain necessary modifications. We begin with part (a). 
Let $s_0$ be as in the proof of Lemma~\ref{l:pivotals}(a). 
Then there exist $a,b\in\mathsf X$ and $r_1\in(0,1)$ such that with positive probability, $a$ and $b$ belong to unbounded occupied components of opposite types, $d_\mathsf X(a,b)\leq s_0+\frac14 r_1$, and all the balls centered in $\mathsf B(a,s_0+1)$ have radius at least $r_1$. 
Let $B=\mathsf B(a,\frac12 r_1)$. As in the proof of Lemma~\ref{l:pivotals}(a), we conclude that on the above event, there are no balls centered in $B$ and for $\mu_\mathsf X$-almost all $x\in B$, adding a ball centered at $x$ with radius bigger than $r_1$ to the occupied set would change the type of one of the components $C_\mathcal O(a)$ or $C_\mathcal O(b)$ to the opposite. The existence of a required $r_*$-pivotal ball $\mathsf B(z,\delta)$ for $C_\mathcal O(0)$  now follows just as in the proof of Lemma~\ref{l:pivotals}(a), where we take $\delta=\frac18r_1$ and arbitrary $r_*$ that satisfies $\mathsf P[r_1<\varrho<r_*]>0$. 

As for part (b), let $s_0$ be as in the proof of Lemma~\ref{l:pivotals}(b). 
Then there exist $a,b\in\mathsf X$ and $r_2$ such that with positive probability, $a$ and $b$ belong to unbounded vacant components of opposite types, $d_\mathsf X(a,b)\leq s_0+\frac14$, and all the balls intersecting $\mathsf B(a,\frac12)$ have radius at most $r_2$ (the existence of such $r_2$ follows from the assumption \eqref{eq:assumption-iid-vacant}). 
As in the proof of Lemma~\ref{l:pivotals}(b), we conclude that for $\mu_\mathsf X$-almost every $x\in \mathsf B(a,\frac12)$, deleting all the balls centered in $\mathsf B(x,\Delta)$, where $\Delta=r_2+1$, would change the type of one of the components $C_\mathcal V(a)$ or $C_\mathcal V(b)$ to the opposite. Finally, we choose $r_*$ such that, in addition to the above, with positive probability, all the balls centered in $\mathsf B(a,\frac12+\Delta)$ have radius at most $r_*$. 
The existence of a required $r_*$-pivotal ball $\mathsf B(z,\Delta)$ for $C_\mathcal V(0)$  now follows just as in the proof of Lemma~\ref{l:pivotals}(b).
\end{proof}

\begin{proof}[Proof of Theorem~\ref{thm:main-iid}]
The proofs of Theorems~\ref{thm:monotonicity-of-uniqueness}--\ref{thm:rw-transience-intro} work the same for Boolean models with i.i.d.\ random radii, so we only need to discuss necessary adjustments to the proof of Theorem~\ref{thm:Indistinguishability}. 
We begin with part (a) of Theorem~\ref{thm:Indistinguishability}. 
In the definition of the event $\mathcal E_0$, one should replace the pivotal ball $\mathsf B(z,\frac18)$ by a $r_*$-pivotal ball $\mathsf B(z,\delta)$, with $\delta$ and $r_*$ as in Lemma~\ref{l:pivotals:iid}(a). 
In the definition of the event $\mathcal E_1$, the ball $\mathsf B(z,\frac98)$ should be replaced by $\mathsf B(z,\delta+r_*)$. 
The set $\mathsf S(z;y)$ in the definition of the event $\mathcal U(z;y)$ should be replaced by the set of all $(x,r)\in\mathsf B(z,\delta)\times[0,r_*]$ such that the connected components of $y$ in $\mathcal O(\widetilde \omega|_{(\mathsf B(z,\delta)\times[0,r_*])^c})$ and $\mathcal O(\widetilde \omega|_{(\mathsf B(z,\delta)\times[0,r_*])^c}+\delta_{(x,r)})$ have different types. 
Note that the so defined set $\mathsf S(z;y)$ is $\sigma(\widetilde \omega|_{(\mathsf B(z,\delta)\times[0,r_*])^c})$-measurable and if $\mathsf B(z,\delta)$ is $r_*$-pivotal for $C_\mathcal O(y)$, then $(\mu_\mathsf X\otimes\mathsf Q)(\mathsf S(z;y))>0$ a.s.\ 
In particular, if $\mathcal U(z;y)$ is the event that $\widetilde\omega(\mathsf S(z;y)) = \widetilde\omega\big(\mathsf B(z,\delta)\times[0,r_*]\big) = 1$, then by Lemma~\ref{l:conditional-probability-positive-iid}(a), there exists $\sigma>0$ such that 
with positive probability, the event $\mathcal E_1$ occurs and the point $z$ in the definition of $\mathcal E_1$ satisfies additionally that 
$\mathsf P\big[\mathcal U(z;y)\,\big|\,\widetilde\omega|_{(\mathsf B(z,\delta)\times[0,r_*])^c}\big]\geq \sigma$.
Modulo these adjustments, the proof of Theorem~\ref{thm:Indistinguishability}(a) for general Boolean models is the same as its proof in the case of constant radii and we omit the details.  

\smallskip

We now discuss necessary adjustments to the proof of Theorem~\ref{thm:Indistinguishability}(b). 
In the definition of the event $\mathcal E_0$, one should replace the pivotal ball $\mathsf B(z,2)$ by a $r_*$-pivotal ball $\mathsf B(z,\Delta)$, with $\Delta$ and $r_*$ as in Lemma~\ref{l:pivotals:iid}(b).
In the definition of $\mathcal E_1$, the ball $\mathsf B(z,3)$ should be replaced by $\mathsf B(z,\Delta+r_*)$. 
By Lemma~\ref{l:conditional-probability-positive-iid}(b), there exists $\sigma>0$ such that 
with positive probability, 
$\mathcal E_1$ occurs and the point $z$ in the definition of $\mathcal E_1$ satisfies additionally that 
$\mathsf P\big[\widetilde\omega(\mathsf B(z,\Delta)\times[0,r_*])=0\,\big|\,\widetilde\omega|_{(\mathsf B(z,\Delta)\times[0,r_*])^c}\big]\geq \sigma$.
Modulo these adjustments, the proof of Theorem~\ref{thm:Indistinguishability}(b) for general Boolean models is the same as its proof in the case of constant radii and we omit the details.  
\end{proof}

\section*{Acknowledgements}
The research of both authors has been supported by the DFG Priority Program 2265 ``Random Geometric Systems'' (Project number 443849139).

\end{document}